\documentclass[12pt, reqno]{amsart}
     \makeatletter
     \def\section{\@startsection{section}{1}%
     \z@{.7\linespacing\@plus\linespacing}{.5\linespacing}%
     {\bfseries
     \centering
     }}
     \def\@secnumfont{\bfseries}
     \makeatother
\usepackage{amssymb}
\usepackage[final]{hyperref}

\usepackage{amsrefs}

\usepackage{tikz}
 
\usetikzlibrary{calc,fadings,decorations.pathreplacing}

\tikzset{%
  >=latex, 
  inner sep=0pt,%
  outer sep=2pt,%
  mark coordinate/.style={inner sep=0pt,outer sep=0pt,minimum size=3pt,
    fill=black,circle}%
}

 \usepackage{wrapfig}

\usepackage{enumerate}

\newtheorem{theorem}{Theorem}[section]
\newtheorem{lemma}{Lemma}[section]
\newtheorem{prop}{Proposition}[section]

\theoremstyle{definition}

\numberwithin{equation}{section}

\usepackage{amssymb}
\usepackage[final]{hyperref}
\usepackage{tikz}
 
 \usepackage{wrapfig}

\newcommand{\mch}{{\mathcal H}}
\newcommand{\mcp}{{\mathcal P}}
\newcommand{\mcz}{{\mathcal Z}}

\newcommand{\mbr}{{\mathbb R}}

\newcommand{\mbc}{{\mathbb C}}

 \newcommand{\ch}{{\mathcal H}}

\newcommand{\ovs}{\overline{\sigma}}

\newcommand{\la}{{\langle}}

\newcommand{\ra}{{\rangle}}

\newcommand{\norm}[1]{\left\lVert#1\right\rVert}

\begin{document}

\title[Polynomials and High-Dimensional Spheres]{Polynomials and High-Dimensional Spheres}
 
\author{Amy Peterson}
\address{Department of Mathematics \\
  University  of Connecticut\\
Storrs, CT 062569 \\
e-mail: \sl amy.peterson@uconn.edu}
 
\author{Ambar N.~Sengupta}
\address{Department of Mathematics \\
  University  of Connecticut\\
Storrs, CT 062569\\
e-mail: \sl ambarnsg@gmail.com}
\thanks{Research   supported in part  by NSA grant H98230-16-1-0330}

\subjclass[2010]{Primary 44A12, Secondary 28C20, 60H40}

\date{11 March 2019}

\dedicatory{}

\begin{abstract} We show that a natural class of orthogonal polynomials on large spheres in $N$ dimensions tend to Hermite polynomials in the large-$N$ limit. We determine the behavior of the spherical Laplacian as well as zonal harmonic polynomials in the large-$N$ limit.  \end{abstract}

\maketitle
 
\section{Introduction}\label{s:Intro}

The uniform measure, of unit total mass, on the sphere of radius $\sqrt{N}$ in $\mbr^N$ is known to converge, in a suitable sense, to Gaussian measure on an infinite-dimensional space $\mbr^\infty$. We show that  orthogonalizing   monomials on the sphere leads, in the large-$N$ limit, to Hermite polynomials.  We prove a similar result for zonal harmonics and we determine the large-$N$ behavior of the spherical Laplacian. Along the way we establish a number of results about spherical harmonics and the spherical Laplacian from a purely algebraic viewpoint. 

The relation between the infinite-dimensional Gaussian and the high-dimension limit of uniform measures on large spheres has been known historically in the study of  the motion of gas molecules by Maxwell  \cite{MaxGas1860}  and Boltzmann  \cite[pages 549-553]{BoltzStud1868}, and later in more mathematical form, in the works of  Wiener  \cite{WienDS1923}, L\'evy \cite{LevyPLAF1922}, McKean \cite{McKGDS1973},  Hida \cite{HidaStat1970}, and Hida and Nomoto \cite{HiNo1964}.     Umemura and Kono \cite{Ume1965}, following up on the works of Hida et al., studied the limiting behavior of spherical harmonics and spherical Laplacians. They also obtained the limiting results on zonal harmonics and the spherical Laplacian, but our algebraic framework  and methods are different.
 A more detailed description of the distinguishing aspects of our approach is given in subsection \ref{ss:dist}. 
  
   \subsection{Summary of results}\label{ss:sumres} Towards the goal of studying the large-$N$ limit of polynomials and the spherical Laplacian on $S^{N-1}(\sqrt{N})$ we also work through an algebraic framework for spherical harmonics. However, spherical harmonics and the spectrum of the spherical Laplacian are classical subjects and it is not our intention here to provide a comprehensive review of these topics.

   We will work mostly with polynomials and sometimes with the corresponding functions.   Let $\mcp$ be the algebra of all polynomials in a sequence of variables $X_1, X_2, \ldots$. The subalgebra of polynomials in $X_1,\ldots, X_N$ will be denoted $\mcp_N$, and those that are homogeneous of degree $d$ by $\mcp^d$.  To pass from polynomials to functions defined on the sphere $S^{N-1}(a)$ of radius $a>0$ we show in Proposition \ref{P:polysphere} that a polynomial $p\in\mcp_N$ that evaluates to $0$ at every point on $S^{N-1}(a)$ is a polynomial multiple of $X_1^2+\ldots +X_N^2-a^2$.  We denote the ideal of all such multiples in $\mcp_N$ by $\mcz_N(a)$.   Next we show in Proposition \ref{P:reprsnt} how to extract from a polynomial $q\in\mcp_N$ a `minimal' representative $q_{\rm min}\in\mcp_k$ modulo the ideal $\mcz_N(a)$; this is, roughly, the smallest $k$ for which $q$ has a representative, mod $\mcz_N(a)$, that depends only on $X_1,\ldots, X_k$. 
   
  For any positive integer $N$,   the space $\mcp_N$ has, for each $a>0$, the sesquilinear pairing  given by 
\begin{equation}\label{E:ippqintro}
\la p, q\ra_{a, N}=\int_{S^{N-1}(a)}p(x)\overline{q(x)}\,d\ovs(x),
\end{equation}
where $\ovs$ is the unit-mass uniform surface measure on the sphere $S^{N-1}(a)$. 
We show in Lemma    \ref{L:zerof} and  subsection \ref{ss:ipkh} that this is an inner-product on $\mcp_k$ if $k<N$. We will often be interested in the sphere of radius $\sqrt{N}$; we use the notation:
\begin{equation}
\la p, q\ra_N=\la p,q \ra_{\sqrt{N}, N}.
\end{equation}

    In Theorem \ref{T:ipNinfty} we show that for any polynomials $p, q\in\mcp_k$ and any $N>k$, we have the limiting inner-product
   \begin{equation}
   \lim_{N\to\infty}\la p, q\ra_N=\la p, q\ra_{L^2(\mbr^\infty, \mu)}.
   \end{equation}
  where on the right hand side we view $p$ and $q$ as functions on the space $\mbr^\infty$ of real sequences, equipped with the product Gaussian measure. 
  

 { In section \ref{s:poly} we use the rotation generating operators 
   $$M_{jk}=X_j\partial_k-X_k\partial_j$$
    acting on  $\mcp$ and define the `quadratic Casimir'

   \begin{equation}\label{E:LapSumRotGen}\norm{{\bf M}}^2=\sum_{\{j,k\}\in P_2(N)}M_{jk}^2,\end{equation}
   where the sum is over $P_2(N)$, the set of all 2-element subsets of $\{1,2,\ldots, N\}$. An algebraic computation establishes the following very useful identity connecting the Euclidean Laplacian $\Delta_N=\sum_{j=1}^N\partial_j^2$ and the quadratic Casimir:
   \begin{equation}\label{E:x2Dintro}
\begin{split}
\norm{X}^2\Delta_{N} &=  (r\partial_r)^2+(N-2)r\partial_r +\norm{{\bf M}}^2,
\end{split}
\end{equation}
where now $\norm{X}^2=X_1^2+\ldots +X_N^2$, and $r\partial_r=\sum_{j=1}^NX_j\partial_j$. 

In section \ref{s:lapl} we will define the spherical Laplacian $\Delta_{S^{N-1}(a)}$ by using the quadratic Casimir $\norm{\bf{M}}^2$; for now let us note that the two are related but not the same.

} 

We introduce in (\ref{P:harmsphere}) a projection map
$$L_a:\mcp_N\to\mch_N=\ker\Delta_N,$$
for which $L_ap$ agrees with $p$ as a function on $S^{N-1}(a)$ (the existence of $L_ap$ is a known result but we present a proof in  subsection \ref{ss:redharm}). We work through algebraic properties of $L_a$ such as its commutativity with $M_{jk}$ and hence with $\norm{{\bf M}}^2$:
$$L_a\norm{{\bf M}}^2=\norm{{\bf M}}^2L_a$$
(This is in (\ref{E:LMjk}) and Proposition \ref{P:spherLapl}) and show (Proposition \ref{P:Laselfadj}) that it is an orthogonal projection relative to the pairing $\la\cdot,\cdot\ra_{a,N}$.

In subsections \ref{ss:sphlapeigen} and \ref{ss:zsh} we briefly study spherical harmonics and zonal harmonics. This is  a classical subject,   but for our needs we develop the results and notion in terms of the algebra of polynomials so that we can apply the results to the situation of interest to us where {\em both the radius and the dimension change}.     In Proposition \ref{P:qXmdeg} we show that if a one-variable polynomial $q_m$  satisfies the differential equation
  \begin{equation}\label{E:legendintro}
(a^2-X^2)q''(X)-(N-1)Xq'(X)+m(m+N-2)q(X)=0
\end{equation}
then  the harmonic polynomial $L_aq_m(X_1)$   is  homogeneous of degree $m$; this is called a {\em zonal harmonic}.

We show in Lemma    \ref{L:zerof} and in subsection \ref{ss:ipkh} the pairing $\la \cdot,\cdot\ra_{a,N}$  is an inner-product   on the subspace $\mch_N\subset\mcp_N$ comprised of harmonic polynomials.   In Proposition \ref{P:adj} we show that the operators $M_{jk}$ are skew-self-adjoint and $\norm{{\bf M}}^2$ is self-adjoint with respect to this inner-product.  The homogeneous  harmonic polynomials of different degrees are orthogonal. 

We show in Proposition \ref{P:orthoqmX} that the zonal harmonic polynomials $q_0, q_1, q_2,\ldots$  arise from the Gram-Schmidt process applied to the inner-product on $\mbc[X]$ given by
\begin{equation}
\la p,q\ra_{a,N}= \frac{c_{N-2}}{c_{N-1}a}\int_{-a}^ap(x)\overline{q(x)} \left(1-\frac{x^2}{a^2}\right)^{\frac{N-3}{2}}\,dx,
\end{equation}
where $c_j$ is the surface area of the $j$-dimensional unit sphere.

In section \ref{s:lapl} we study the limiting behavior of the Laplacian on the sphere $S^{N-1}(\sqrt{N})$ as $N\to\infty$. To this end we give a precise definition of the spherical Laplacian
$$\Delta_{S^{N-1}(a)}$$
on the polynomial algebra $\mcp_N$.  There is a subtlety here: for any $p\in\mcp_n$, for $n<N$, we define $\Delta_{S^{N-1}(a)}p$ to be the unique polynomial in $\mcp_n$ for which
$$\Delta_{S^{N-1}(a)}p=\frac{1}{a^2}\sum_{\{j,k\}\in P_2(N)}M_{jk}^2p\qquad\hbox{mod $\mcz_n(a)$.}$$
That such a unique polynomial exists is shown in Proposition \ref{P:Mjkfmin}. We also prove that then  $\Delta_{S^{N-1}(a)}$ is a linear operator, self-adjoint on each finite-dimensional space $\mcp^d_N$. 

We prove in Proposition \ref{P:limspjLapl} that the spherical Laplacian $\Delta_{S^{N-1}(\sqrt{N})}$ converges to the Hermite differential operator:
\begin{equation}
\lim_{N\to\infty} \Delta_{S^{N-1}(\sqrt{N})}=\sum_{j=1}^\infty\left[\partial_j^2-X_j\partial_j\right].
\end{equation}
It is important to note that here on the left the limit is in the pointwise sense for a sequence of operators on the vector space of  polynomials.

We also show that a  limiting relationship between Gegenbauer polynomials and Hermite polynomials, given by
 \begin{equation}\label{E:GgnHerm0}
\lim_{N\to\infty}C^{\left((N-2)/2\right)}_m(X/\sqrt{N})= H_m(X),
\end{equation}
 can be explained in terms of the large-$N$ limit of zonal harmonics for $S^{N-1}(\sqrt{N})$.

 \subsection{Distinctive features of our approach}\label{ss:dist}   Let us note some significant aspects of our work that differentiate it from other related works: (i)  we make extensive use of  the rotation generator operators $M_{jk}$ and establish numerous results, such as those in subsection \ref{ss:mjkharm},  for the action of these operators on polynomials; (ii) our use of the quadratic Casimir $\norm{\mathbf M}^2$, especially the identity (\ref{E:x2Dintro}); (iii) the introduction of the restriction map $L_a$ in Proposition \ref{P:harmsphere}; (iv) we prove a large number of algebraic results of intrinsic interest, such as 
 Proposition \ref{P:polysphere} and Proposition \ref{P:rotinvpoly2}, as well as many other results such as the limiting orthogonal projection result Proposition \ref{ss:proj};   (v) our definition of the spherical Laplacian $\Delta_{S^{N-1}(a)}$  in subsection \ref{ss:sphlap}, which depends critically on the uniqueness result Proposition \ref{P:reprsnt}.  For fixed $N$ one could define the spherical Laplacian to be the same as the quadratic Casimir, but in studying the behavior with varying $N$ it is necessary to be more careful.
 Our result Proposition \ref{P:limspjLapl} on the limiting behavior of the spherical Laplacian $\Delta_{S^{N-1}(\sqrt{N})}$ is, as best as we can determine, logically different from the similar result in \cite{Ume1965} because of the precise definition of the spherical Laplacian operator that we use. Finally, another difference, in framework, between \cite{Ume1965} and our work is that we work with ordinary limits of results on $\mbr^N$ and do not need to use projective limit spaces to state our results.

\subsection{Notation} We will summarize here much of the notation used in the paper.  Let $\mcp$ be the algebra of all polynomials in the sequence of variables $X_1, X_2, \ldots$. Then we have the algebraic direct sum decomposition
\begin{equation}\label{E:ppd}
\mcp=\mcp^0\oplus\mcp^1\oplus\mcp^2\oplus\ldots,
\end{equation}
where
\begin{equation}
\hbox{$\mcp^d$ is the space of polynomials homogeneous of degree $d$.}
\end{equation}
For example, $X_1^3X_7^2+X_2X_5X_9^3\in\mcp^5$. Let us note that
$$\mcp^0=\mbc.$$

Let
\begin{equation}\label{E:mcpleqd}\mcp^{\leq d}=\mcp^0\oplus\mcp^1\oplus\ldots \oplus\mcp^d.
\end{equation}
 be the subspace of polynomials of total degree $\leq d$.  The subspace of polynomials in $X_1,\ldots, X_N$ will be denoted by $\mcp_N$:
\begin{equation}
\mcp_N=\mbc[X_1,\ldots, X_N].
\end{equation}
Moreover, we use the notation:
 \begin{equation}
 \begin{split}
 \mcp^d_N &=\mcp_N\cap\mcp^d\\
 \mcp^{\leq d}_N &=\mcp_N\cap \mcp^{\leq d}.
 \end{split}
 \end{equation}
 A basis of $\mcp_N^d$ is formed by all the monomials
 $$X_1^{j_1}\ldots X_N^{j_N}$$
 for which  each $j_a$ is a non-negative integer and
 $$j_1+\ldots +j_N= d.$$
We note that for such a monomial and any $n\in\{1,\ldots, N-1\}$, the monomials $X_1^{j_1}\ldots X_n^{j_n}$ and $X_{n+1}^{j_{n+1}}\ldots X_N^{j_{N }}$ are both in  the same monomial basis of $\mcp_N^{\leq d}$.

The subspace of $\mcp_N$ consisting of all polynomials homogeneous of degree $d$ will be denoted $\mcp^d_N$; thus
\begin{equation}
\mcp_N=\bigoplus_{d\geq 0}\mcp^d_N
\end{equation}
as an algebraic direct sum. The Laplacian $\Delta_N$ is the operator on $\mcp_N$ given by
\begin{equation}
\Delta_N=\partial_1^2+\ldots +\partial_N^2,
\end{equation}
where $\partial_j$ is the usual partial differential operator in the variable $X_j$.
We denote  by $\mch_N$ the subspace of all harmonic polynomials in $\mcp_N$; thus 
\begin{equation}\label{E:defHN}
\mch_N=\ker\Delta_N.
\end{equation}
  If $p\in\mch_N$ then we can write $p$ in a unique way as the sum of harmonic polynomials that are homogeneous of different degrees; thus, we have the algebraic direct sum:
   \begin{equation}
   \mch_N=\mch^0_N\oplus \mch^1_N\oplus \ldots,
   \end{equation}
   where $\mch^d_N$ is the subspace consisting of all harmonic polynomials  that are homogeneous of degree $d$. We also use the notation
   \begin{equation}
   \mch^{\leq d}_N=\mch^0_N\oplus\ldots\oplus\mch^d_N.
   \end{equation}

\section{Spherical and Gaussian integration}\label{s:ifh}

In this section we establish relationships between integration over spheres and integration with respect to  Gaussian measure.

\subsection{Integration of homogeneous functions}\label{ss:inth} A function $f$ on $\mbr^N$ is said to be {\em homogeneous of degree $d$} if
$$f(tx)=t^df(x)\qquad\hbox{for all $x\in\mbr^N$ and $t\in\mbr$.}$$
We will work with homogeneous polynomial functions.

Let us note that the product of a homogeneous function of degree $d_1$ and a homogeneous function of degree $d_2$ is a homogeneous function of degree $d_1+d_2$.  

\begin{prop}\label{P:ip}
Let $f$ be a Borel function on $\mbr^N$,  homogeneous of degree $d$ and integrable with respect to the standard Gaussian measure. Then
\begin{equation}\label{E:ip}
2^ {\frac{d}{2}}\frac{  \Gamma\left(\frac{d+N}{2}\right) }{  \Gamma\left(\frac{ N}{2}\right) }\int_{S^{N-1}} f\,d\ovs
=\int_{\mbr^N} f(x) (2\pi)^{-N/2}e^{-\frac{\norm{x}^2}{2}}\,dx,
\end{equation}
where on the left $\ovs$ is the uniform measure on the unit-sphere $S^{N-1}$, normalized to having total measure $1$, and on the right we have the standard Gaussian measure on $\mbr^N$.
\end{prop}
Let us note for us later that by homogeneity of $f$ we have
\begin{equation}\label{E:ipa}
 \int_{S^{N-1}(a)} f\,d\ovs
=a^d \frac{ \Gamma\left(\frac{ N}{2}\right)  }{ 2^ {\frac{d}{2}} \Gamma\left(\frac{d+N}{2}\right)}
\int_{\mbr^N} f(x) (2\pi)^{-N/2}e^{-\frac{\norm{x}^2}{2}}\,dx,
\end{equation}
for any radius $a>0$.

\begin{proof}  We have the polar disintegration formula

\begin{equation}\label{E:polg}
\int_{\mbr^N} f(x) (2\pi)^{-N/2}e^{-\frac{\norm{x}^2}{2}}\,dx=\int_0^\infty\left[\int_{S^{N-1}(r)}f(x)\,d\sigma(x)\right]\,(2\pi)^{-N/2}e^{-r^2/2}dr,
\end{equation}
where $\sigma$ is the standard surface measure on the sphere $S^{N-1}(r)$ of radius $r$ (see, for example, \cite{SenGRL2018}*{(3.10)} for proof). Then we observe by homogeneity of $f$ that the spherical integral over the sphere of radius $r$ is a multiple of the integral over the unit sphere:
\begin{equation}\label{E:Pg}
\begin{split}
\int_{S^{N-1}(r)}f(x)\,d\sigma(x) & =\int_{S^{N-1}} f(rx)\,r^{N-1}\,d\sigma(x)\\
&=r^{d+N-1}\int_{S^{N-1}}f(x)\,d\sigma(x).
\end{split}
\end{equation}
Using this in the Gaussian integration (\ref{E:polg}) we have
\begin{equation}\label{E:polg2}
\begin{split}
\int_{\mbr^N} f(x) (2\pi)^{-N/2}e^{-\frac{\norm{x}^2}{2}}\,dx &=\int_0^\infty\left[r^{d+N-1}\int_{S^{N-1}} f\,d\sigma \right]\,(2\pi)^{-N/2}e^{-r^2/2}dr\\
&=\left[\int_{S^{N-1}} f\,d\sigma\right]\int_0^\infty r^{d+N-1}(2\pi)^{-N/2}e^{-r^2/2}\,dr\\
&=(2\pi)^{-N/2}2^{\frac{d+N  }{2}-1} \Gamma\left(\frac{d+N}{2}\right)\int_{S^{N-1}} f\,d\sigma,
\end{split}
\end{equation}
where in obtaining the Gamma function we used the substitution $y=r^2/2$ in the integration.
Taking $f=1$ (degree $0$) here gives the surface area of the sphere $S^{N-1}$ to be
\begin{equation}\label{E:surf}
c_{N-1}=2\frac{\pi^{N/2}}{ \Gamma(N/2)}.
\end{equation}
Then, returning to (\ref{E:polg2}), we have:
\begin{equation}\label{E:polg3}
\begin{split}
&\int_{\mbr^N} f(x) (2\pi)^{-N/2}e^{-\frac{\norm{x}^2}{2}}\,dx \\
& =(2\pi)^{-N/2}2^{\frac{d+N  }{2}-1} \Gamma\left(\frac{d+N}{2}\right)c_{N-1}\int_{S^{N-1}} f\,d\ovs\\
&=2^{\frac{d}{2}}\frac{  \Gamma\left(\frac{d+N}{2}\right) }{  \Gamma\left(\frac{ N}{2}\right) }\int_{S^{N-1}} f\,d\ovs
\end{split}
\end{equation}
 
\end{proof}

As an immediate consequence of this result, we see that the Gaussian $L^2$  inner-product when applied to homogeneous  functions can be computed by working out the  $L^2$ inner-product on the unit sphere, with normalized uniform measure:

\begin{prop}\label{P:ipsph} Let $f$ and $g$ be homogeneous Borel functions on $\mbr^N$,  of degrees $d_f$ and $d_g$, respectively, and square-integrable with respect to standard Gaussian measure $\mu$. Then 
\begin{equation}\label{E:ipfg}
\la f, g\ra_{L^2(\mbr^N, \mu)}=2^{ d}\frac{  \Gamma\left(\frac{N}{2}+d\right) }{  \Gamma\left(\frac{ N}{2}\right) } \la f, g\ra_{L^2(S^{N-1}, \ovs)},
\end{equation}
where $d=(d_f+d_g)/2$. If $d_f+d_g$ is odd then both sides in (\ref{E:ipfg}) are $0$. We also have 
 \begin{equation}\label{E:ipfg3}
 \begin{split}
\la f, g\ra_{L^2(\mbr^N, \mu)} 
&=a_{d,N} \la f, g\ra_{L^2(S^{N-1}(\sqrt{N}),\ovs)}
\end{split}
\end{equation}
where
\begin{equation}\label{E:adN}
a_{d,N}=\prod_{j=1}^d\left(1+2\left(\frac{j-1}{N}\right)\right). 
\end{equation}
\end{prop}

\begin{proof} The product $f{\bar g}$ is homogeneous of degree $2d$. Then, applying Proposition \ref{P:ip} to the function  $f{\bar g}$ we obtain the formula (\ref{E:ipfg}). If $d_f+d_g$ is odd then $(fg)(-x)=-fg(x)$ for all $x\in\mbr^N$, and so both sides of (\ref{E:ipfg}) are $0$.

Let us then assume that $d_f+d_g$ is even; then $d$ is an integer. Suppose $d\geq 1$. Since
$$\Gamma(s+1)=s\Gamma(s),$$
we have
$$\Gamma(s+d)=(s+d-1)(s+d-2)\ldots s\Gamma(s),$$
and so, with $s=N/2$ we have
\begin{equation}\label{E:GamN}
 \Gamma\left(\frac{N}{2}+d\right)=2^{-d}(N+2d-2)(N+2d-4)\ldots N\Gamma(N/2).
 \end{equation}
 Then the formula  (\ref{E:ipfg}) becomes
 \begin{equation}\label{E:ipfg2}
\la f, g\ra_{L^2(\mbr^N, \mu)}= (N+2d-2)(N+2d-4)\ldots N\la f, g\ra_{L^2(S^{N-1}, \ovs)}.
\end{equation}
If $d=0$ then this equation follows directly from (\ref{E:ipfg}), provided we interpret the right hand side as just $\la f, g\ra_{L^2(S^{N-1}, \ovs)}$.

Using homogeneity of $f$ and $g$ again, we can rewrite the right hand side as an integral over the sphere of radius $\sqrt{N}$:
 \begin{equation}\label{E:ipfg4}
 \begin{split}
&\la f, g\ra_{L^2(\mbr^N, \mu)}\\
 &= \left(1+2\left(\frac{d-1}{N}\right)\right)\left(1+2\left(\frac{d-2}{N}\right)\right)\ldots *1*N^d\la f, g\ra_{L^2(S^{N-1}, \ovs)}\\
&=a_{d,N} \la f, g\ra_{L^2(S^{N-1}(\sqrt{N}),\ovs)}
\end{split}
\end{equation}
where $a_{d,N}$ is as given in (\ref{E:adN}). \end{proof}

   We now show that the pairing
\begin{equation}\label{E:ipN}
\la p, q\ra_{a, N}=\int_{S^{N-1}(a)} p(x)\overline{q(x)}\,d\ovs(x),
\end{equation}
where $ \ovs$ is the unit-mass uniform measure on $S^{N-1}(a)$, gives an inner-product on $\mcp_k$, for any positive integers $k<N$.

\begin{lemma}\label{L:zerof} Let $p\in\mcp_k$ be a polynomial in $k$ variables, let $f$ be the function on $\mbr^{N}$, where $N>k$,  given by
$$f(x_1,\ldots, x_N)=p(x_1,\ldots, x_k)\qquad\hbox{for all $(x_1,\ldots, x_N)\in\mbr^N$.}$$
If $f=0$ on $S^{N-1}(a)$, where $a>0$, then the polynomial $p$ is $0$. The  inner-product $\la\cdot,\cdot\ra_{a, N}$ restricts to an inner-product on the subspace $\mcp_k$ for $k<N$.
\end{lemma}
\begin{proof} Let $(x_1,\ldots, x_k)$ be any point within the open ball of radius $a$ in $\mbr^k$. Then 
$$p(x_1,\ldots, x_k)= f\bigl(x_1,\ldots, x_k, \sqrt{a^2-\norm{x}^2},0,\ldots, 0)=0,$$
where $\norm{x}^2=\sqrt{x_1^2+\ldots +x_k^2}$.
A polynomial function in $\mbr^k$ that is zero on an open set is identically zero, and so $p$ is the zero polynomial.

If $\norm{p}_{a, N}^2$ is $0$ then the evaluation of $p$ at every $x\in S^{N-1}(a)$ is zero, and so  $p$ is the zero polynomial. {Therefore $\la\cdot,\cdot\ra_{a, N}$ restricts to an inner-product on the subspace $\mcp_k$ for $k<N$.} \end{proof}

\subsection{Gaussian integration as a limit of spherical integration}\label{ss:sphgau} There is a well-known relationship between Gaussian integration in infinite dimensions and integration over large spheres (see, for example,  \cite{SenGRL2016, SenGRL2018}). Here we focus on the special case of polynomial functions.  

We use the product Gaussian measure $\mu$ on $\mbr^\infty$, the space of all real  sequences. The measure $\mu$ is supported on much smaller subspaces of $\mbr^\infty$, but we do not need to bring such subspaces in for our purposes here.

\begin{theorem}\label{T:ipNinfty}
Suppose $p$ and $q$ are    polynomial functions on $\mbr^k$, viewed also as functions on $\mbr^N$ for $N>k$ in terms of the first $k$ coordinates. Then
\begin{equation}\label{E:ipfg4a}
 \lim_{N\to\infty} \la p, q\ra_{L^2(S^{N-1}(\sqrt{N}),\ovs)}=  \la p, q\ra_{L^2(\mbr^\infty, \mu)},
\end{equation}
with notation as before.
\end{theorem}
Recall that by Lemma \ref{L:zerof}, $\la \cdot,\cdot\ra_{L^2(S^{N-1}(a))}$ is an inner-product on the space of polynomials in $X_1,\ldots, X_k$, for any $k<N$. So, restricting to the case of $p, q\in\mcp^d_k$, the result (\ref{E:ipfg4a}) says that the inner-product $\la \cdot,\cdot\ra_{L^2(S^{N-1}(a))}$ converges to the Gaussian inner-product on $\mcp^d_k$ for all  integers $d\geq 0$ and $k\geq 1$.
\begin{proof} 
By linearity we may assume that $p$ and $q$ are homogeneous, since a general polynomial is a sum of homogeneous monomials. Then using the identity (\ref{E:ipfg3}) and observing that $\lim_{N\to\infty}a_{d,N}=1$ we obtain (\ref{E:ipfg4a}).
\end{proof}

\section{Harmonic polynomials and operators on polynomial spaces}\label{s:poly}

We   work with the polynomial spaces $\mcp^d_k$. As noted in the introduction, the subject of harmonic polynomials is classical and there are many treatments of the subject available, both old and recent (Stein and Weiss \cite[Chapter 4]{SW1971} or  Kellogg \cite[Chapter V]{Kell1929}, to name just two). However, for our objectives it is necessary to develop the subject in a way that does not focus on functions on a fixed sphere in a fixed dimension. For this reason we develop the ideas and results in terms of the algebra of polynomials.

\subsection{Decomposition of polynomials using harmonic components}\label{ss:harm}  The results of this subsection are well-known, but we include them here both for ease of reference and to help set up the framework we will be using later. 

Let $\la\cdot,\cdot\ra_h$ be the `Hermite inner-product' on the space $\mcp$ of polynomials, specified by
\begin{equation}\label{E:herip}
\la X_1^{j_1}\ldots X_m^{j_m}, X_1^{k_1}\ldots X_m^{k_m}\ra_h = j_1!\ldots j_m!\delta_{jk},
\end{equation}
where $\delta_{jk}$ is $1$ if $j$ and $k$ are equal as finite sequences and is $0$ otherwise. The reason we call this the Hermite inner-product is because
\begin{equation}
\begin{split}
&
\la X_1^{j_1}\ldots X_m^{j_m}, X_1^{k_1}\ldots X_m^{k_m}\ra_h 
\\
&=\la  H_{j_1}(X_1)\ldots H_{j_m}(X_m), H_{k_1}(X_1)\ldots H_{k_m}(X_m)\ra_{L^2(\mu)}, 
\end{split}
\end{equation}
where on the right we use the Gaussian inner-product from Theorem \ref{T:ipNinfty} and $H_n(X)$ is the standard Hermite polynomial. In more detail, $H_n(X)$ is the component of $X^n$ that is $L^2(\mu)$-orthogonal to the subspace spanned by $1, X, X^2, \ldots, X^{n-1}$.

 Then the operator
$ \partial_j$, the derivative with respect to $X_j$, has the adjoint given by multiplication by $X_j$:
\begin{equation}
\la \partial_jp, q\ra_h=\la p, X_jq\ra_h,
\end{equation}
for all polynomials $p$ and $q$, as can be readily verified by taking $p$ and $q$ to be monomials. Then for the Laplacian
$$\Delta_N=\sum_{j=1}^N\partial_j^2$$
the adjoint is obtained from:
\begin{equation}\label{E:LaplNh}
\la \Delta_Np, q\ra_h=\la p, \norm{X}^2q\ra_h
\end{equation}
for all $p, q\in\mcp_N$. Thus the adjoint of $\Delta_N$ is given by multiplication by 
$$ \norm{X}^2=X_1^2+\ldots+X_N^2.$$
Using this adjoint relation we have the following standard result:
\begin{prop}\label{P:splitharm}
For any integers $N\geq 1$ and $d\geq 2$, the space $\mcp^d_N$ of polynomials in $X_1,\ldots, X_N$ that are homogeneous of degree $d$ splits as a direct sum:
\begin{equation}\label{E:mcpNmchN}
\mcp^d_N=\mch^d_N\oplus \norm{X}^2\mcp^{d-2}_N,
\end{equation}
with these subspaces being orthogonal with respect to the Hermite inner-product $\la\cdot,\cdot\ra_h$.  Moreover,
 \begin{equation}\label{E:splitmchleq}
 \mcp^{\leq d}_N=\mch^{\leq d}_N\oplus \norm{X}^2\mcp^{\leq d-2}_N.
 \end{equation}

\end{prop}
\begin{proof}  The decomposition (\ref{E:mcpNmchN}) follows by applying the linear algebra result Lemma \ref{L:splitvT} below with $T=\Delta_N:\mcp^d_N\to\mcp^{d-2}_N$. The decomposition (\ref{E:splitmchleq}) follows as a consequence or by applying Lemma \ref{L:splitvT} below with $T=\Delta_N:\mcp^{\leq d}_N\to\mcp^{\leq d-2}_N$.
\end{proof}

  We have used the following  elementary result in linear algebra.  
 \begin{lemma}\label{L:splitvT}
 Suppose $T:V\to W$ is a linear mapping between finite-dimensional inner-product spaces. Then
 $TT^*:W\to W$ maps ${\rm Im}(T)$ isomorphically onto itself. Moreover, for any $v\in V$ we have
  \begin{equation}\label{E:splitvT}
 v =[v-T^*(TT^*)^{-1}Tv] \,+\, T^*(TT^*)^{-1}Tv,
 \end{equation}
 where the first term on the right hand side is in $\ker T$ and the second term is in ${\rm Im}(T^*)$. The subspaces $\ker T$ and ${\rm Im}(T^*)$ are mutually orthogonal and their direct sum is $V$:
 \begin{equation}\label{E:kerTsum}
 V= (\ker T)\oplus {\rm Im}(T^*).
 \end{equation}
 \end{lemma}
 \begin{proof} Let us first check that the term $(TT^*)^{-1}Tv$ on the right hand side makes sense. Since the image of  $TT^*$ is contained in ${\rm Im}(T)$ it follows that it maps this subspace into itself. Since we are working in finite dimensions, $(TT^*)$ will map ${\rm Im}(T)$ surjectively onto itself if   $TT^*$ is injective when restricted to ${\rm Im}(T)$. We verify this injectivity.  If $Tv\in \ker(TT^*)$ then 
$$\norm{T^*Tv}^2=\la T^*Tv, T^*Tv\ra=\la TT^*(Tv), Tv\ra =\la 0, Tv\ra=0,$$
and so $T^*Tv=0$, which then implies
$$\norm{Tv}^2=\la Tv, Tv\ra=\la T^*Tv,v\ra=0.$$
Thus, if $Tv\in\ker({TT^*})$ then $Tv=0$. So  {$TT^*$} maps ${\rm Im}(T)$ isomorphically onto itself. Hence, the terms on the right hand side of (\ref{E:splitvT}) are all meaningful. The remaining statements are all readily verified.
  \end{proof}

 Thus,
  every    polynomial $p$ can be `divided' by $\norm{X}^2$ to leave a remainder that is harmonic:
\begin{equation}\label{E:pp0q1}
p=p_0+\norm{X}^2q,
\end{equation}
  where $p_0\in\mch_N$ and $q\in\mcp_{ N}$, with the degree of $q$ being less than the degree of $p$.
  
  \begin{prop}\label{P:harmdiv}
  For any $p\in\mcp_N$ there is a unique   $p_0\in\mch_N$ and a unique $q\in \mcp_N$ such that
  \begin{equation}
  p=p_0+\norm{X}^2q.
  \end{equation}
  \end{prop}
  \begin{proof} We have already seen the existence of $p_0$ and $q$. For uniqueness, suppose that $q\in\mcp_N$ is such that $\norm{X}^2q$ is harmonic; then by the adjoint relation (\ref{E:LaplNh}) we have
  \begin{equation}
  \la\norm{X}^2q, \norm{X}^2q\ra_h =\la \Delta_N(\norm{X}^2q), q\ra_h=\la 0, q\ra_h=0.
  \end{equation}
  Thus $\norm{X}^2q=0$ and so $q=0$.

  \end{proof}

  Inductively, we can `expand' $p\in \mcp^{\leq d}_N$ in `base' $\norm{X}^2$:
  \begin{equation}\label{E:pexpand}
  p=p_0+\norm{X}^2p_1+\ldots+ \norm{X}^{2s}p_{s},
  \end{equation}
 with $s$ being the largest integer $\leq d/2$, where $p_0, p_1,\ldots, p_s$ are harmonic  polynomials, with $p_k\in \mcp^{\leq d-2k}_N$.

\begin{prop}\label{P:harmsphere}
Given a polynomial function $p$ on $\mbr^N$ and a sphere in $\mbr^N$, there is a unique harmonic polynomial $L_ap$ which coincides as a function with $p$ on the sphere. The polynomial $L_ap$ is given by 
\begin{equation}\label{E:pstar}
L_ap=p_0+a^2p_1+\ldots+a^{2s}p_s,
\end{equation}
where the notation is as in (\ref{E:pexpand}). The degree of $L_ap$ is $\leq $ the degree of $p$.
\end{prop}
If $p=\norm{X}^2$ then $L_ap$ is the constant $a^2$. Thus the degree of $L_ap$ need not be equal to the degree of $p$.

  \begin{proof}  The polynomial $L_ap$  is harmonic because each term on the right hand side in (\ref{E:pstar})  is a harmonic polynomial. Moreover, by (\ref{E:pexpand}), 
$$\hbox{$L_ap(x)=p(x)$  for all $x\in S^{N-1}(a)$.}$$
Thus $L_ap$ is a harmonic polynomial that agrees with $p$ pointwise on the sphere $S^{N-1}(a)$. 

For the uniqueness statement, suppose $p_*$ is a harmonic polynomial  that agrees pointwise with $p$ on the sphere $S^{N-1}(a)$, for some $a>0$.  Then $L_ap-p_*$ is a harmonic polynomial that evaluates to $0$ at each point of $S^{N-1}(a)$. 
  By a result from analysis, a harmonic function that vanishes on the boundary of a ball is equal to zero in the interior of the ball. Thus   $L_ap-p_*$ evaluates to zero in the interior of the ball of radius $a$, and hence is the zero polynomial. (Any polynomial  in $N$ variables that evaluates to $0$ at all points in an open ball is identically zero; this can be proved by induction on $N$ and the fact that any nonzero one-variable polynomial has finitely many zeros.)
  
   Let $p$ be of degree $d$. Then from (\ref{E:pexpand}) and (\ref{E:pstar}) we see that  the degree of $L_ap$ is $\leq d$, being equal to $d$ if and only if $p_0\neq 0$; indeed, if $p_0\neq 0$ then the part of $p$ that is homogeneous of highest degree is the same as that of $L_ap$.
    \end{proof}

  \begin{prop}\label{P:polysphere} If a polynomial $p\in \mbc[X_1,\ldots, X_N]$ evaluates to $0$ at all points on $S^{N-1}(a)$, for some $a>0$, then there is a polynomial $q\in  \mbc[X_1,\ldots, X_N]$  such that $p=(X_1^2+\ldots +X_N^2-a^2)q$.
  \end{prop} 
  \begin{proof} By definition, $L_ap$ is the harmonic polynomial that agrees with $p$ pointwise on $S^{N-1}(a)$; thus if $p$ evaluates to $0$ on this sphere then so does $L_ap$.  Consequently (by the argument in the proof of  Proposition \ref{P:harmsphere}), $L_ap=0$. Then
  \begin{equation}
  p=p-L_ap =(\norm{X}^2-a^2)p_1  +(\norm{X}^4-a^4)p_2+\ldots (\norm{X}^{2r}-a^{2r})p_r,
  \end{equation} 
  which is a polynomial multiple of $\norm{X}^2-a^2$.  \end{proof}
  
  Keith Conrad has shared with us an entirely algebraic proof of this result.
  
  We denote by ${\mathcal Z}_N(a)$ the ideal in $\mbc[X_1,\ldots, X_N]$ consisting of all polynomial multiples of $(X_1^2+\ldots +X_N^2-a^2)$. Moreover, let
  \begin{equation}
  \mcz^{\leq d}_N(a)=\mcz_N(a)\cap \mcp^{\leq d}_N,
  \end{equation}
  the subspace of $\mcz_N(a)$ consisting of all elements of degree $d$. Here we  continue to work with a fixed radius $a>0$.
  
  \begin{prop}\label{P:reprsnt}
  Suppose $p\in \mbc[X_1,\ldots, X_N]$ and $a>0$. Then either $p+\mcz_N(a)$ contains no nonzero polynomial in   $X_1,\ldots, X_k$, for $k<N$,  or there is a unique  polynomial $p_{\rm min}$ in   variables $X_1, \ldots, X_k$ with $k<N$, such that $p_{\rm min}\in p+\mcz_N(a)$.  \end{prop}
If $p+\mcz_N(a)$ contains no nonzero polynomial in   $X_1,\ldots, X_k$, for $k<N$ then we take $p_{\rm min}$ to be just $p$ itself.
  \begin{proof} Suppose $p_1$ and $p_2$ are both polynomials in  $X_1,\ldots, X_j$, with $j<N$, for which $p_1, p_2\in p+\mcz_N(a)$. Then, by Proposition \ref{P:polysphere},  $p_1-p_2$ would be divisible by $\norm{X}^2-a^2$ in $\mbc[X_1, \ldots, X_N]$. But no nonzero polynomial multiple of $\norm{X}^2-a^2$ is a polynomial in fewer variables than $X_1,\ldots, X_N$, because for any polynomial $q$ we have
  $$(\norm{X}^2-a^2)q=X_1^2q+\ldots +X_N^2q-a^2q,$$
  which is of degree at least $2$ in each $X_j$ unless $q$ is $0$.
  Hence $p_1=p_2$.
  \end{proof}

  \subsection{The rotation generators $M_{jk}$  and their action on  $\mcp^{\leq d}_N/\mcz^{\leq d}_N(a)$. }  For    $j, k\in\{1,\ldots, N\}$, let
  \begin{equation}
  M_{jk}=X_j\partial_k-X_k\partial_j.
  \end{equation}
  Then 
  $$M_{kj}=-M_{jk}.$$
  The following result produces an action of $M_{jk}$ on each space $\mcp^{\leq d}_N/\mcz^{\leq d}_N(a)$:
  \begin{prop}\label{P:Mjkquot}
  The operator $M_{jk}$ maps each space $\mcp^d_N$ into itself and it also maps $\mcz^{\leq d}_N(a)$ into itself.
  \end{prop}
  \begin{proof} Clearly $M_{jk}$ preserves degrees and so maps $\mcp^d_N$ into itself.  Now consider $p\in\mcz_N(a)$. Then
  $$p=(X_1^2+\ldots +X_N^2-a^2)q$$
  for some $q\in\mcp^{d-2}_N$.  Then, using the Leibniz formula for first-order differential operators, we have
  \begin{equation}
  M_{jk}p =\left[M_{jk}(X_1^2+\ldots +X_N^2-a^2)\right]q+(X_1^2+\ldots +X_N^2-a^2)M_{jk}q.
  \end{equation}
  The first term on the right hand side is readily checked to be $0$. Hence $M_{jk}p$ is   in $\mcz_N(a)$.
   \end{proof}
   Thus we have an induced linear map
   \begin{equation}
   M_{jk}: \mcp^{\leq d}_N/\mcz^{\leq d}_N(a) \to  \mcp^{\leq d}_N/\mcz^{\leq d}_N(a).
   \end{equation}
   
   \subsection{Reduction to harmonic polynomials}\label{ss:redharm} We have seen that the restriction of any polynomial $p$ to $S^{N-1}(a)$ coincides with the restriction of a (unique) harmonic polynomial $L_ap$  to the sphere. In terms of the expansion (\ref{E:pexpand}) the mapping $L_a$ is given by:
   \begin{equation}\label{E:L}
  L_a: \mcp^{\leq d}_N\to \mcp^{\leq d}_N : p\mapsto p_0+a^2p_1+\ldots +a^{2s}p_s,
   \end{equation}
   where each $p_j$ is harmonic of degree $d-2j$. 
   
   \begin{prop}\label{P:propLa}
   The image of $L_a$ is the space of harmonic polynomials of degree $\leq d$:
   \begin{equation}
   L_a(\mcp_N^{\leq d})=\mch_N^{\leq d}.
   \end{equation}
   The operator $L_a$ is a projection:
   \begin{equation}
   L_a^2=L_a
   \end{equation}
   and its kernel is
   \begin{equation}\label{E:kerLa}
   \ker L_a=\mcz_N^{\leq d}.
   \end{equation}
    \end{prop}
    We note that then $L_a$ induces an isomorphism
     \begin{equation}
   \mcp^{\leq d}_N/\mcz^{\leq d}_N(a)\simeq \mch^{\leq d}_N.
   \end{equation}
   In Proposition \ref{P:Laselfadj}  we will show that $L_a$ is self-adjoint with respect to a pairing $\la\cdot,\cdot\ra_{a,N}$. Thus, it is the orthogonal projection in $\mcp^{\leq d}_N$ onto the subspace of harmonic polynomials.

\begin{proof} By definition,
    $L_ap$ is the harmonic polynomial that agrees with $p$ on $S^{N-1}(a)$. In particular, $L_ap=p$ if $p\in\mch^{\leq d}_N$, and so the image of $L_a$ is $\mch^{\leq d}_N$.    Moreover, 
    $$L_a^2p=L_a(L_ap)=L_ap,$$
    since $L_ap=p$ if $p$ is already harmonic.   
    
    Observing that
   \begin{equation}
   p=L_ap+(\norm{X}^2-a^2)p_1+\ldots +(\norm{X}^{2s}-a^{2r})p_s,
   \end{equation}
   we see that $\ker L_a\subset \mcz^{\leq d}_N(a)$. Moreover, if $p\in\mcz^{\leq d}_N(a)$ then, since $L_ap$ agrees with $p$ on $S^{N-1}(a)$, we have $L_ap=0$ on $S^{N-1}(a)$ and hence, since $L_ap$ is harmonic, $L_ap$ is the zero polynomial, and so $p\in\ker L_a$. This proves that $\ker L_a$ is $\mcz^{\leq d}_N$.
   \end{proof}

   \subsection{Rotation generators  on harmonic polynomials}\label{ss:mjkharm} We show now that the operators $M_{jk}$  commute with the projection $L_a$ and map harmonic polynomials to themselves. 
     
     \begin{prop}\label{P:mjkharm} With notation as above,  
     \begin{equation}\label{E:XnormMjk}
  \norm{X}^2M_{jk}=M_{jk}\norm{X}^2,
   \end{equation}
   where $\norm{X}^2$ is the operator that multiplies any polynomial by $\norm{X}^2$, and
   
\begin{equation}\label{E:MjkDeltaN}
   \Delta_N M_{jk}=M_{jk}\Delta_N.
   \end{equation}
     The operator $M_{jk}$ maps $\mch^d_N$ into itself.  Moreover,
    \begin{equation}\label{E:LMjk}
   L_aM_{jk}=M_{jk}L_a.
   \end{equation}
      \end{prop}
   \begin{proof}  Since $M_{jk}$ is a first-order differential operator we have by the Leibniz product rule:
   \begin{equation}
   M_{jk}(\norm{X}^2p)=M_{jk}(\norm{X}^2)p+\norm{X}^2M_{jk}p=\norm{X}^2M_{jk}p,
   \end{equation}
   because 
   $$M_{jk}(X_j^2+X_k^2)=0.$$
   Thus (\ref{E:XnormMjk}) holds.

   Straightforward algebraic computation shows that (\ref{E:MjkDeltaN}) holds, and hence $M_{jk}p$ is harmonic whenever $p$ is harmonic.
   
   Applying $M_{jk}$ to the expansion 
    \begin{equation}\label{E:pexpand2}
  p=p_0+\norm{X}^2p_1+\ldots+n \norm{X}^{2r}p_{r}
  \end{equation}
and using the Leibniz product rule $M_{jk}(fg)=(M_{jk}f)g+f(M_{jk}g)$ and the fact that $M_{jk}\norm{X}^2=0$, we have
   \begin{equation}\label{E:Mjkp}
   M_{jk}p=M_{jk}p_0 +\norm{X}^2M_{jk}p_1+\ldots +\norm{X}^{2r}M_{jk}p_r.
   \end{equation}
   Here each of the terms $M_{jk}p_i$ is harmonic and of the same degree as $p_i$. Hence, (\ref{E:Mjkp}) is the expansion of $M_{jk}p$ in `base' $\norm{X}^2$, the counterpart of  (\ref{E:pexpand2}) for $M_{jk}p$. Then, by the  definition of $L_a$ in (\ref{E:L}) we have:
   \begin{equation}
   L_aM_{jk}p= M_{jk}p_0 +a^2M_{jk}p_1+\ldots + a^{2r}M_{jk}p_r.
   \end{equation}
   Consequently, 
   \begin{equation}\label{E:LaMjk}
   L_aM_{jk}p=M_{jk}L_ap.
   \end{equation}
   \end{proof}

   \subsection{The quadratic Casimir $\norm{\mathbf M}^2$}\label{ss:quadCas} We define the quadratic Casimir operator on polynomials by
    \begin{equation}\label{E:DSN}
   \norm{\mathbf M}^2= \sum_{\{j,k\}\in P_2(N) }M_{jk}^2,
    \end{equation}
    where $P_2(N)$ is the set of all $2$-element subsets of $\{1,\ldots, N\}$.     This operator  acts on the full polynomial algebra $\mcp$ in the sequence of variables $X_1, X_2, \ldots$, but we will often restrict it  to appropriate subspaces.
        
    \begin{prop}\label{P:spherLapl} $\norm{\mathbf M}^2$ maps each of  the spaces $\mcp^d_N$, $\mcz^{\leq d}_N(a)$, and  $\mch^{  d}_N$  into itself. Moreover,  
    \begin{equation}\label{E:DSLcom}
    \begin{split}
    M_{jk}\norm{{\bf M}}^2&= \norm{{\bf M}}^2M_{jk}\\
 \norm{\mathbf M}^2L_a &=L_a \norm{\mathbf M}^2.
    \end{split}
    \end{equation}
    for all $a>0$.   \end{prop} 
\begin{proof}  By Proposition \ref{P:Mjkquot},  each rotation generator $M_{jk}$ maps each of the spaces $\mcp^d_N$ and $\mcz^{\leq d}_N(a)$ into itself.  Hence,  $\norm{{\bf M}}^2$ also maps each of these spaces into itself. The commutation of $L_a$ and $\norm{{\bf M}}^2$ follows from the commutation (\ref{E:LMjk}) relation between $L_a$ and all $M_{jk}$.  By Proposition \ref{P:mjkharm}, $M_{jk}$ maps  $\mch^d_N$ into itself. Hence  $\norm{{\bf M}}^2$ maps  $\mch^d_N$ into itself. 

The first commutation relation in (\ref{E:DSLcom}) follows by combining the identity (\ref{E:x2D}), expressing $\norm{{\bf M}}^2$ in terms of $\Delta_N$ and $r\partial_r$ (see (\ref{E:rdr})) with the commutation relation  (\ref{E:MjkDeltaN})  between $M_{jk}$ and $\Delta_N$,  the commutation relation (\ref{E:XnormMjk}) between $\norm{X}^2$ and $M_{jk}$, and the commutation relation between $M_{jk}$ and $r\partial_r$ given by Lemma \ref{L:MjkL}.  
\end{proof}

\subsection{ $\norm{{\bf M}}^2$ and the Euclidean Laplacian $\Delta_N$}\label{ss:lseuc} We turn to the relationship between $\norm{{\bf M}}^2$ and the ordinary Laplacian $\Delta_N$ on $\mcp_N$:
\begin{equation}
\Delta_N=\sum_{j=1}^N\partial_j^2.
\end{equation}
We use the notation
\begin{equation}\label{E:rdr}
r\partial_r\stackrel{\rm def}{=}\sum_{j=1}^NX_j\partial_j,
\end{equation}
as an operator on $\mcp_N$. This notation is consistent with
$$r^2=\norm{X}^2=X_1^2+\ldots+X_N^2.$$
Then
\begin{equation}\label{E:rdrsqrd}
(r\partial_r)^2= \sum_{j=1}^NX_j^2\partial_j^2  +\sum_{j=1}^NX_j\partial_j +2\sum_{\{j,k\}\in P_2(N)}X_jX_k\partial_j\partial_k,
\end{equation}
where   $P_2(N)$ be the set of all 2-element subsets   $\{j,k\}$ of $\{1,\ldots, N\}$.  
For $j\neq k$, we have
\begin{equation}
M_{jk}^2= X_j^2\partial_k^2 +X_k^2\partial_j^2-\left[X_j\partial_j+ X_k\partial_k\right]-2X_jX_k\partial_j\partial_k.      \end{equation}

Since each $j$ appears paired exactly once with each of the other  $(N-1)$ choices of $k$, we have
\begin{equation}\label{E:smjk}
\begin{split}
&\sum_{\{j,k\}\in P_2(N)}M_{jk}^2 \\
&=  \sum_jX_j^2\left(\sum_k\partial_k^2 -\partial_j^2\right) -(N-1)\sum_{j=1}^N X_j\partial_j\\
&\qquad  - 2\sum_{\{j,k\}\in P_2(N)} X_j X_k\partial_j\partial_k\\
&= \norm{X}^2\Delta_{N} - \sum_{j=1}^N\left[X_j^2\partial_j^2+(N-1)X_j\partial_j\right]
 \\
&\qquad  - 2\sum_{\{j,k\}\in P_2(N)} X_jX_k\partial_j\partial_k.\end{split}
\end{equation}
Thus
\begin{equation}\label{E:x2D}
\begin{split}
\norm{X}^2\Delta_{N} &= \sum_{j=1}^N X_j^2\partial_j^2     +2\sum_{\{j,k\}\in P_2(N)}X_j X_k\partial_j\partial_k \\
&\quad +(N-1)\sum_j X_j\partial_j +\sum_{\{j,k\}\in P_2(N)}M_{jk}^2 \\
&=(r\partial_r)^2+(N-2)r\partial_r +\norm{{\bf M}}^2.
\end{split}
\end{equation}
where we used (\ref{E:rdrsqrd}) in the last step. This relation holds on the full polynomial space $\mcp_N$.

 \subsection{Eigenvalues of $r\partial_r$}\label{ss:rprreigen} If $q\in\mcp^m_N$, that is a polynomial in $X_1,\ldots, X_N$ that is homogeneous of degree $m$, then we have the Euler relation
\begin{equation}\label{E:euler}
(r\partial_r)q=mq,
\end{equation}
 as can be readily verified by taking $q$ to be any monomial. Since eigenvectors belonging to distinct eigenvalues are linearly independent,  and every polynomial is the sum of its homogeneous parts, the eigenvalues of $r\partial_r$ are all the non-negative integers. By the same argument, the eigenvalues of the operator $(r\partial_r)^2+(N-2)r\partial_r$ are $m(m+N-2)$, with $m$ running over all non-negative integers. Since these numbers are distinct for distinct values of $m$, such eigenvectors are linearly independent. Thus in the direct sum decomposition
 \begin{equation}\label{E:eigprpr2}
 \mcp_N=\mcp_N^0\oplus\mcp^1_N\oplus\ldots
 \end{equation}
 each subspace $\mcp^m_N$ is an eigenspace of  $(r\partial_r)^2+(N-2)r\partial_r$, with eigenvalue $m(m+N-2)$.

\begin{lemma}\label{L:MjkL}
The operators $M_{jk}$ and $r\partial_r$ commute for all $\{j,k\}\in \mcp_2(N)$.
\end{lemma}
\begin{proof} The Euler relation (\ref{E:euler}) shows that $r\partial_r$ is constant on the space $\mcp^d_N$ of polynomials homogeneous of degree $d$, and the operator $M_{jk}$ maps $\mcp^d_N$ into itself. Hence these two operators commute on each space $\mcp^d_N$, and hence on all of $\mcp_N$.
\end{proof}

 \subsection{Eigenvectors for $\norm{\mathbf M}^2$}\label{ss:sphlapeigen}  Suppose $p$ is a  harmonic polynomial, homogeneous of degree $m$, in $N$ variables.  Then, using (\ref{E:x2D}), we have  
\begin{equation}\label{E:eigenD}
\begin{split}
   \norm{{\bf M}}^2p  & = r\Delta_{N}p - (r\partial_r)^2p  -(N-2)(r\partial_r)p \\
&=0 -m^2p    -(N-2) mp\\
&= -m(m+N-2)p.
\end{split}
\end{equation}
  Conversely, if $p$ is a polynomial, homogeneous of degree $m$ in $N$ variables, and if
\begin{equation}\label{E:eigenD2}
  \norm{{\bf M}}^2p = -m(m+N-2)p 
\end{equation}
then $p$ is harmonic:
$$\Delta_{ N}p =0.$$
Thus the operator $\norm{{\bf M}}^2$, when restricted to $\mch_N$, has eigenvalue spectrum given by the distinct numbers $ -m(m+N-2)$, as $m$ runs over all non-negative integers, and the corresponding decomposition of $\mch_N$ into eigenspaces is
\begin{equation}
\mch_N=\mch^0_N\oplus\mch^1_N\oplus\mch^2_N\oplus\ldots .
\end{equation}

\subsection{Polynomials with rotational symmetry}\label{ss:rotsym} 

We will use the notation
\begin{equation}
AX=\left(\sum_{j=1}^Na_{1j}X_j, \ldots, \sum_{j=1}^Na_{Nj}X_j\right).
\end{equation}
for $A=[a_{ij}]$ any $N\times N$ matrix with entries $a_{ij}\in\mbc$.

Next we determine the nature of polynomials invariant under all rotations and then those that are invariant under rotations that fix a particular direction.

\begin{prop}\label{P:rotinvpoly} Suppose $p(X_1,\ldots, X_N)\in\mbc[X_1,\ldots, X_N]$ is invariant under all rotations, in the sense that
\begin{equation}\label{E:pinvso}
\begin{split}
p(X_1,\ldots, X_N)&=p(AX)\\
&\hbox{for all $A=[a_{ij}] \in SO(N)$.}
\end{split}
\end{equation}
If $N>1$ then there is a polynomial $q(X)\in\mbc[X]$ such that
\begin{equation}\label{E:pq21}
p(X_1,\ldots, X_N)=q(X_1^2+\ldots+X_N^2).
\end{equation}

\end{prop}
If $N=1$ and (\ref{E:pinvso}) holds for all $A\in O(1)$ then (\ref{E:pq21}) holds; this is simply the statement that if  a polynomial $p(X)$ is even in the sense that $p(-X)=p(X)$ then all the terms in $p(X)$ with odd powers of $X$ have coefficient $0$.
\begin{proof} We argue by induction on the total degree $d_p$ of $p(X_1,\ldots, X_N)$ (if this polynomial is zero then the result is trivially true). If $d_p=0$ then $p(X_1,\ldots, X_N)$ is just a constant and so (\ref{E:pq21}) holds automatically by taking the right hand side to also be the constant polynomial.

  The polynomial $p_0= p(X_1,\ldots, X_N)-p(1,0,\ldots, 0)$ evaluates to $0$ at all points on the unit sphere $S^{N-1}$. Hence  by Proposition \ref{P:polysphere}  $p_0$ is a polynomial multiple of $\norm{X}^2-1$:
\begin{equation}\label{E:pp1}
p(X_1,\ldots, X_N)=p(1,0,\ldots, 0)+\left(\norm{X}^2-1\right)p_1(X_1,\ldots, X_N)
\end{equation}
where $p_1(X_1,\ldots, X_N)
\in\mbc[X_1,\ldots, X_N]$.  Then for any $A\in SO(N)$ we have
\begin{equation}
 p(X_1,\ldots, X_N)=p(AX)=p(1,0,\ldots, 0)+\left(\norm{AX}^2-1\right)p_1(AX)
\end{equation}
Since $\norm{AX}^2=\norm{X}^2$, this implies
\begin{equation}
(\norm{X}^2-1)\left(p_1(AX) - p_1(X)\right)=0.
\end{equation}
The ring $\mbc[X_1,\ldots, X_N]$ has no zero divisors, it follows that
$$ p_1(X_1,\ldots, X_N)=p_1(AX) \qquad\hbox{for all $A\in SO(N)$.}$$
We see from (\ref{E:pp1}) that $p_1(X_1,\ldots, X_N)$  is either $0$ or has total degree less than the total degree of $p(X_1,\ldots, X_N)$. Thus, inductively, $p_1(X_1,\ldots, X_N)$ is a polynomial in $\norm{X}^2$. Substituting in (\ref{E:pp1}) we see then that $p(X_1,\ldots, X_N)$ is also a polynomial in $\norm{X}^2$. \end{proof}

\begin{prop}\label{P:rotinvpoly2} Suppose $p(X_1,\ldots, X_N)\in\mbc[X_1,\ldots, X_N]$ is invariant under all rotations that fix a particular nonzero vector $t=(t_1,\ldots, t_N)\in\mbr^N$, in the sense that
\begin{equation}\label{E:pinvsot}
\begin{split}
p(X_1,\ldots, X_N)&=p(AX)\\
&\hbox{for all $A \in SO(N)$ for which $At=t$.}
\end{split}
\end{equation}
Then   there is a polynomial $q(Y, Z)\in\mbc[Y,Z]$ such that for any $a>0$,
\begin{equation}\label{E:pq2}
p(x_1,\ldots, x_N)=q(a, t_1x_1+\ldots+t_Nx_N)\qquad\hbox{for all $(x_1,\ldots, x_N)\in S^{N-1}(a)$.}
\end{equation}
Conversely, if a polynomial $p$ is of the form (\ref{E:pq2}) and if $p$ is homogeneous then $p$ satisfies the rotational invariance (\ref{E:pinvsot}).
\end{prop}
 \begin{proof} Any nonzero vector can be transformed into any other nonzero vector by rotation and scaling. With this in mind, we can assume that $t$ is the vector $(1,0,\ldots, 0)$. Then the hypothesis is that 
 $$\hbox{$p(X_1, A(X_2,\ldots, X_N))=p(X_1,\ldots, X_N)$ for all rotations $A\in SO(N-1)$. }$$
   Writing
  \begin{equation}\label{E:ppm}
  p(X_1,\ldots, X_N)=p_mX_1^m+\ldots +p_1X_1+p_0,
  \end{equation}
 where each $p_j$ is a polynomial in $(X_2,\ldots, X_N)$, we conclude then that each of these coefficient polynomials $p_j$ is invariant under the action of $SO(N-1)$ on $(X_2,\ldots, X_N)$. 
Then, by Proposition  \ref{P:rotinvpoly}, each $p_j$ is a polynomial in $X_2^2+\ldots+X_N^2$:
\begin{equation}\label{E:pjqj}
p_j=q_j(X_2^2+\ldots+X_N^2),
\end{equation}
where $q_j$ is a polynomial in one variable. 

For $a>0$ and any point $x=(x_1,\ldots, x_N)\in S^{N-1}(a)$ the rotational symmetry of $p$ implies that
\begin{equation}\label{E:ppx10}
p(x)=p(x_1, \sqrt{a^2-x_1^2}, \, 0,\ldots, 0).
\end{equation}
Then by (\ref{E:ppm}) and (\ref{E:pjqj}) we have
\begin{equation}\label{E:pqmx}
p(x)= q_m(a^2-x_1^2)x_1^m+\ldots +q_1(a^2-x_1^2)x_1+p_0
\end{equation}
With this in mind let us take
\begin{equation}
q(Y,Z)=Z^mq_m(Y^2-Z^2)+\ldots +Zq_1(Y^2-Z^2)+q_0.
\end{equation}
Then by (\ref{E:pqmx}) we have:
\begin{equation}
p(x)=q(a, x_1),
\end{equation} 
for all $x=(x_1,\ldots, x_N)\in S^{N-1}(a)$.

For the converse statement we note that if $p$ has the form (\ref{E:pq2}) then $p$ as a function on the sphere $S^{N-1}(a)$ is invariant under rotations preserving $(t_1,\ldots, t_N)$. Since $p$ is homogeneous, it is uniquely determined by its values on the sphere. Hence $p$ is invariant under all such rotations.
\end{proof}

\subsection{Zonal spherical harmonics}\label{ss:zsh}  By a  {\em zonal} harmonic we mean a harmonic polynomial $p$ that is invariant under all rotations that preserve a particular direction.  As before we take this direction to be $(1,0,\ldots, 0)$.  Then, by Proposition \ref{P:rotinvpoly2}, $p$, as a function on any sphere $S^{N-1}(a)$ of radius $a>0$, agrees pointwise with a one-variable polynomial $q(a; X_1)$.  
 We will compute $\norm{{\bf M}}^2$ on $q(a; X_1)$.  First we have
 \begin{equation}
 \begin{split}
 M_{1j}(X_1^m)&=X_1*0 - X_jmX_1^{m-1}=-mX_1^{m-1}X_j, \end{split}
 \end{equation}
 and so
  \begin{equation}\label{E:M1j2X1}
 \begin{split}
 M_{1j}^2(X_1^m)&=-mX_1^m+m(m-1)X_1^{m-2}X_j^2 \end{split}
 \end{equation}
 which leads to
 \begin{equation}
 \begin{split}
\sum_{\{1,j\}\in P_2(N)}M_{1j}^2(X_1^m) &= \sum_{j=2}^N M_{1j}^2(X_1^m) \\
&= -m(N-1)X_1^m+ m(m-1)\left(\sum_{j=2}^NX_j^2\right)X_1^{m-2}\\
&=  -m(N-1)X_1^m+ m(m-1)\left(\sum_{j=1}^NX_j^2-X_1^2\right)X_1^{m-2}.
\end{split}
\end{equation}
That is,
  \begin{equation}\label{E:M1j2X2}
 \begin{split}
\norm{\mathbf M}^2(X_1^m)&= -m(N-1)X_1^m+ m(m-1)\left(\sum_{j=1}^NX_j^2-X_1^2\right)X_1^{m-2}.\end{split}
 \end{equation}
The map $L_a$ in (\ref{E:L}) carries any polynomial  multiple of $\norm{X}^2-a^2$ to zero, and so
\begin{equation}\label{E:sphLxm}
 L_a \norm{\mathbf M}^2X_1^m = -m(N-1)L_a(X_1^m)+m(m-1)L_a\left((a^2-X_1^2)X_1^{m-2}\right).
\end{equation}
Then for any polynomial $q(X_1)$ we have:
\begin{equation}\label{E:a2Ds}
 L_a\norm{{\bf M}}^2q(X_1) = L_a\left[-(N-1) X_1q'(X_1) +  (a^2-X_1^2)q''(X_1)\right].
\end{equation}
 Let us recall that $L_a[q(X_1)]$ is the harmonic polynomial in $X_1,\ldots, X_N$ that agrees with $q$ as a function on the sphere $S^{N-1}(a)$.  Now we can show that $L_aq(X_1)$  is homogeneous.

\begin{prop}\label{P:qXmdeg} Suppose $q(a; X)$ is a   polynomial in $X$  that satisfies
\begin{equation}\label{E:legend}
(a^2-X^2)q''(a;X)-(N-1)Xq'(a;X)+m(m+N-2)q(a;X)=0,
\end{equation}
where $N$ is a positive integer and $m$ is a non-negative integer.
Then the harmonic polynomial $L_a[q(a;X_1)]$ that coincides with $q(a;X_1)$ pointwise on the sphere $S^{N-1}(a)$  is  homogeneous of degree $m$. Conversely, if $q(a; X)$ is a   polynomial in $X$  such that $L_a[q(a;X_1)]$ is homogeneous of degree $m$ then (\ref{E:legend}) holds.
\end{prop}
When $a$ is fixed we will usually just write $q(X)$ instead of $q(a;X)$. In our main application, though, $a=\sqrt{N}$.
\begin{proof} Since $L_a[q(X_1)]$ is a harmonic polynomial, 
using the identity  (\ref{E:x2D}) we have:
\begin{equation}\label{E:norDX1}
\begin{split}
 0 
=&\norm{X}^2\Delta_{N}\bigl(L_a[q(X_1)]\bigr)\\
&\stackrel{\rm by  (\ref{E:x2D})}{=} \left((r\partial_r)^2+(N-2)r\partial_r +\norm{{\bf M}}^2\right)L_a[q(X_1)]\\
&= \left((r\partial_r)^2+(N-2)r\partial_r\right)L_a[q(X_1)]  +L_a\norm{{\bf M}}^2q(X_1)\\
&\hskip 2in\hbox{(on using the second equation in  {(\ref{E:DSLcom})})}\\
&=   \left((r\partial_r)^2+(N-2)r\partial_r\right)L_a[q(X_1)]  \\
&\hskip 1in +L_a\left[-(N-1) X_1q'(X_1) +  (a^2-X_1^2)q''(X_1)\right].
\end{split}
\end{equation}
Thus
\begin{equation}\label{E:norDX}
\begin{split}
&  \left((r\partial_r)^2+(N-2)r\partial_r\right)L_a[q(X_1)]\\
&\hskip 1in =  -L_a\left[-(N-1) X_1q'(X_1) +  (a^2-X_1^2)q''(X_1)\right].
\end{split}
\end{equation}

If $q$ satisfies the differential equation (\ref{E:legend}) then  \begin{equation}
 \left((r\partial_r)^2+(N-2)r\partial_r\right)L_a[q(X_1)] = m(m+N-2)L_a[q(X_1)],
 \end{equation}
which says that $L_a[q(X_1)]$ is an eigenvector of  the operator 
$$\left((r\partial_r)^2+(N-2)r\partial_r\right)$$
with eigenvalue $m(m+N-2)$. By the discussion in the context of (\ref{E:eigprpr2}) we conclude that $L_a[q(X_1)]$ is homogeneous of degree $m$.

Conversely, suppose that the polynomial $L_a[q(X_1)]$ is homogeneous of degree $m$. Then
$$(r\partial_r)L_a[q(X_1)]=mL_a[q(X_1)]$$
and so
$$ \left((r\partial_r)^2+(N-2)r\partial_r\right)L_a[q(X_1)]=m(m+N-2) L_a[q(X_1)].$$
 Then from (\ref{E:norDX}) we see that
 \begin{equation}\label{E:dqkLa}
(a^2-X_1^2)q''(X_1)-(N-1) X_1q'(X_1) + m(m+N-2) q(X_1) \in \ker L_a.
\end{equation}
 Then by Proposition \ref{P:propLa} 
 $$(a^2-X_1^2)q''(X_1)-(N-1) X_1q'(X_1) + m(m+N-2) q(X_1)\in \mcz^{\leq d}_N(a), $$
  and this means (by (\ref{E:kerLa})) that it is a polynomial multiple of $\norm{X}^2-a^2$. Now no nonzero polynomial multiple of $\norm{X}^2-a^2$ can be a polynomial in only $X_1$, unless $N=1$. Thus $q$ satisfies the differential equation (\ref{E:legend}) if $N>1$. When $N=1$ the relation (\ref{E:dqkLa}) means that $m(m-1)q(X_1)$ is a polynomial multiple of $X_1^2-a^2$. Recall that $m$ is the degree of the harmonic polynomial $L_a[q(X_1)]$; but in dimension $N=1$ the only harmonic polynomials are of degree $\leq 1$, and so $m(m-1)=0$ in this case. Thus, even when $N=1$, the left side of (\ref{E:dqkLa}) is $0$.
 \end{proof}
 
\subsection{Examples of zonal harmonics}\label{ss:exzon} Let us look at a few examples. For degree $2$, we have the decomposition of $X_1^2$ into a harmonic part and a multiple of $\norm{X}^2$:
 \begin{equation}
 X_1^2 =X_1^2-\frac{1}{N}\norm{X}^2\,+\, \frac{1}{N}\norm{X}^2,
 \end{equation}
 and so
 \begin{equation}
 L_aX_1^2= X_1^2-\frac{1}{N}\norm{X}^2\,+\, \frac{1}{N}a^2,
 \end{equation}
 and then we see that
 \begin{equation}
 L_a\left(X_1^2-\frac{1}{N}a^2\right)=  X_1^2-\frac{1}{N}\norm{X}^2.
 \end{equation}
 The point here is that the polynomial on the right hand side is homogeneous (as in Proposition \ref{P:qXmdeg}). 
 Next, for degree $3$, we have the decomposition
 \begin{equation}
 X_1^3 = X_1^3-\frac{3}{N+2}\norm{X}^2X_1\,+\, \frac{3}{N+2}\norm{X}^2X_1,
 \end{equation}
 where the first two terms on the right hand side form a harmonic polynomial; then
 \begin{equation}\label{E:LaX13}
 L_a\left(X_1^3-\frac{3}{N+2}a^2X_1\right)= X_1^3-\frac{3}{N+2}\norm{X}^2X_1. \end{equation}
 Again, we have found a homogeneous  degree $3$ harmonic polynomial that, when restricted to the sphere $S^{N-1}(a)$, is invariant under rotations around the $X_1$-axis.
 
 \subsection{Zonal harmonics and Gegenbauer polynomials}\label{ss:gegenzon} 
The Gegenbauer differential equation is
\begin{equation}\label{E:Ggn}
(1-y^2)p''(y) -(2b+1)yp'(y) +m(m+2b)p(y)=0.
\end{equation}
We take the Gegenbauer polynomial $C^{(b)}_m(y)$ to be the degree-$m$  monic polynomial that satisfies this differential equation. If $p$ is monic of degree $n$ then the coefficient of $y^n$ in $(1-y^2)p''(y)-(2b+1)yp'(y)$ is 
$$-n(n-1)-(2b+1)n=-n(n+2b),$$
which shows that $n$ must be equal to $m$ for $p$ to be a solution of (\ref{E:Ggn}).

To match (\ref{E:Ggn})  with (\ref{E:legend}) we take
$$p(y)= q(a; ay)$$
Then
$$p'(y)=aq'(a;ay)\qquad\hbox{and}\qquad p''(y)=a^2q''(a;ay)$$
and so
\begin{equation}
\begin{split}
&(1-y^2)p''(y) -(2b+1)yp'(y) +m(m+2a)p(y)\\
&=(a^2-a^2y^2)q''(a;ay)-(2b+1)ayq'(a;ay)+m(m+2b)q(a;ay)\\
&=(a^2-x^2)q''(a;x)-(2b+1)xq'(a;x)+m(m+2b)q(a;x),
\end{split}
\end{equation}
where $x=ay$. Setting $2b+1=N-1$ it follows from (\ref{E:legend}) that $p(y)$ does satisfy the differential equation (\ref{E:Ggn}). Thus
\begin{equation}\label{E:qmaXGgn}
q_m(a;X) =C^{((N-2)/2)}_m(X/a).
\end{equation}
These are the polynomials that coincide on the sphere   $S^{N-1}(a)$ with the restrictions of harmonic polynomials.

\subsection{The inner-product $\la\cdot,\cdot\ra_{a,N}$}\label{ss:ipkh}  Consider the sesquilinear pairing on $\mcp^d_N$ given by
\begin{equation}\label{E:ippq}
\la p, q\ra_{a, N}=\int_{S^{N-1}(a)}p(x)\overline{q(x)}\,d\ovs(x).
\end{equation}
This satisfies all the conditions for an inner-product except that some nonzero vectors might have norm $0$; these are exactly those polynomials that evaluate to $0$ on $S^{N-1}(a)$. If $p\in\mch^d_N$ satisfies $\la p, p\ra_{S^{N-1}(a)}=0$ then $p$ evaluates to $0$ at all points on $S^{N-1}(a)$ and so, since $p$ is harmonic, it is the zero polynomial. Thus,
$\la\cdot,\cdot\ra_{a, N}$  {\em is an inner-product on}  $\mch^{ d}_N$ for every $d\geq 0$, and hence on $\mch_N$, the space of all harmonic polynomials in $\mcp_N$.

We have seen in Lemma \ref{L:zerof} that $\la\cdot,\cdot\ra_{a,N}$ {\em is an inner-product on the space $\mcp_k$ of polynomials in $X_1,\ldots, X_k$ where} $1\leq  k<N$.

\subsection{Adjoint operators}\label{ss:adj} Using the relationship between integration over spheres and Gaussian integration we determine the adjoints of differential operators over spheres.

\begin{prop}\label{P:adj}
For any polynomials $p, q\in\mcp_N$, and any $j, k\in \{1,\ldots, N\}$,
\begin{equation}\label{E:Mjkadj}
\int_{S^{N-1}(a)}  (M_{jk}p)(x)q(x)\,d\ovs(x)=-\int_{S^{N-1}(a)}  p(x)(M_{jk}q)(x) \,d\ovs(x),
\end{equation}
for any radius $a>0$.  Moreover,
\begin{equation}\label{E:DSadj}
\int_{S^{N-1}(a)}  (\norm{{\bf M}}^2p)(x)q(x)\,d\ovs(x)= \int_{S^{N-1}(a)}  p(x) (\norm{{\bf M}}^2q)(x) \,d\ovs(x).
\end{equation}
The operators $M_{jk}$ are   skew-hermitian and the operator $\norm{{\bf M}}^2$ is hermitian on the  spaces $\mcp_m$, equipped with the inner-product $\la\cdot,\cdot\ra_{a,N}$, for any integer $m$ with $1\leq m<N$:
\begin{equation}\label{E:skewherm}
\begin{split}
\la p, M_{jk}q\ra_{a, N} &=-\la M_{jk}p, q\ra_{a,N} \\
\la p, \norm{{\bf M}}^2q\ra_{a, N} &= \la \norm{{\bf M}}^2p, q\ra_{a,N}
\end{split}
\end{equation}
for all $p, q\in \mcp_m$, polynomials in the variables $X_1,\ldots, X_m$ for $m<N$.

On the space $\mch_N$, equipped with the inner-product $\la\cdot,\cdot\ra_{a,N}$, the operators $M_{jk}$ are skew-hermitian and the operator $\norm{{\bf M}}^2$ is hermitian: that is, the relations (\ref{E:skewherm}) hold also for $p, q\in\mch_N$.  \end{prop}

\begin{proof} Without loss of generality and for simplicity of notation, let $j=1$ and $k=2$. Both sides of (\ref{E:Mjkadj}) are bilinear in $(p,q)$, and so we may assume that $p$ and $q$ are both monomials, and, in particular, homogeneous. Then, integrating over any sphere of radius $a>0$ and using (\ref{E:ipa}) we have
\begin{equation}\label{E:pqint}
\begin{split}
&\int_{S^{N-1}(a)}p(x)\left(x_1\partial_{x_2}-x_2\partial_{x_1}\right)q(x)\,d\ovs(x) \\
&= a^d \frac{ \Gamma\left(\frac{ N}{2}\right)  }{ 2^ {\frac{d}{2}} \Gamma\left(\frac{d+N}{2}\right)}
(2\pi)^{-N/2}\int_{\mbr^N}p(x)\left[\left(x_1\partial_{x_2}-x_2\partial_{x_1}\right)q(x)\right]e^{-\norm{x}^2/2}\,dx
\end{split}
\end{equation}
where $d$ is the sum of the homogeneity degrees of $p$ and $q$. Integrating by parts and simplifying the algebra, the right hand side of (\ref{E:pqint}) is equal to:
\begin{equation}
- a^d \frac{ \Gamma\left(\frac{ N}{2}\right)  }{ 2^ {\frac{d}{2}} \Gamma\left(\frac{d+N}{2}\right)}
(2\pi)^{-N/2}\int_{\mbr^N}q(x)\left[\left(x_1\partial_{x_2}-x_2\partial_{x_1}\right)p(x)\right]e^{-\norm{x}^2/2}\,dx.
\end{equation}
Now we reverse the process and use (\ref{E:ipa}) again to rewrite this last expression as an integral over the sphere $S^{N-1}(a)$, to obtain
$$-\int_{S^{N-1}(a)}q(x)\left(x_1\partial_{x_2}-x_2\partial_{x_1}\right)p(x)\,d\ovs(x) .$$
Thus,
\begin{equation}
\begin{split}
&\int_{S^{N-1}(a)}p(x)\left(x_1\partial_{x_2}-x_2\partial_{x_1}\right)q(x)\,d\ovs(x) \\
&= -\int_{S^{N-1}(a)}q(x)\left(x_1\partial_{x_2}-x_2\partial_{x_1}\right)p(x)\,d\ovs(x).
\end{split}
\end{equation}
This proves (\ref{E:Mjkadj}). Applying this twice, and using  the definition of $\norm{{\bf M}}^2$ we obtain (\ref{E:DSadj}). 

The statements about skew-adjointness of $M_{jk}$ and self-adjointness of $\norm{{\bf M}}^2$ follow by recalling that $\la\cdot,\cdot\ra_{a, N}$ is an inner-product on $\mch_m$, for $m<N$, and that $\la\cdot,\cdot\ra_{a, N}$ is an inner-product on $\mch_N$.
\end{proof}
 
 \begin{prop}\label{P:orthopq} The subspaces $\mch^m_N$ and $\mch^n_N$ in $\mch_N$ are orthogonal, with respect to the inner-product $\la\cdot,\cdot\ra_{a,N}$, if $m\neq n$.
 \end{prop}
 \begin{proof} The operator  $\norm{{\bf M}}^2$ is self-adjoint on $\mch_N$, with respect to $\la\cdot,\cdot\ra_{a,N}$, and has distinct eigenvalues ($-m(m+N-2)$ and $-n(n+N-2)$) on $\mch^m_N$ and on $\mch^n_N$ if $m\neq n$. Hence these subspaces are orthogonal. \end{proof}
 
 Finally, let us note that the operator $L_a$ is also self-adjoint:
 
  \begin{prop}\label{P:Laselfadj} 
  The operator $L_a:\mcp_N\to \mcp_N$ is self-adjoint in the sense that
  \begin{equation}\label{E:Lasa}
  \la L_ap,q\ra_{a,N}=\la p, L_aq\ra_{a,N}
  \end{equation}
  for all $p, q\in\mcp_N$, and all $a>0$.
  
  \end{prop}
  \begin{proof}
  Let $p, q\in\mcp_N$; then
   \begin{equation}
   \begin{split}
   \la L_ap,q\ra_{a,N}&=\int_{S^{N-1}(a)}[L_ap](x)\overline{q(x)}\,d\ovs(x)\\
   &=\int_{S^{N-1}(a)}p(x)\overline{q(x)}\,d\ovs(x)\\
   &\qquad\hbox{(because $L_ap(x)=p(x)$ for $x\in S^{N-1}(a)$)}\\
   &=\int_{S^{N-1}(a)}p(x)\overline{[L_aq](x)}\,d\ovs(x)\\
   &=\la p, L_aq\ra_{a,N},
\end{split}
   \end{equation}
   where again in the last line we used the fact that $L_aq$ and $q$ agree pointwise on $S^{N-1}(a)$.
   \end{proof}

 \subsection{Orthogonality of the zonal harmonic polynomials}\label{ss:zorth}  We recall from subsection \ref{ss:ipkh} that $\la\cdot,\cdot\ra_{a,N}$ is an inner-product on $\mcp_{1}$, the  space of polynomials in $X_1$.  Let $q_m(X)$ be a degree $m$  polynomial for which $L_aq_m(X_1)\in\mcp_N$ is homogeneous of degree $m$; we take $q_0(X_1)$ to be $1$.  To be definite, we take $q_m(X_1)$ to be monic, with the highest degree term being $X_1^m$.
 
 Then, by  Proposition \ref{P:orthopq}  and the fact that $L_aq(X_1)$ is harmonic, we see that
 \begin{equation}\label{E:Laq1ip}
 \la L_aq_m(X_1), L_aq_n(X_1)\ra_{a, N}=0\qquad\hbox{if $m\neq n$}.
 \end{equation}
 Now recall that $L_aq_m(X_1)$ is equal to $q_m(X_1)$ as functions when restricted to the sphere $S^{N-1}(a)$. Hence,   (\ref{E:Laq1ip}) is equivalent to:
   \begin{equation}\label{E:Laq1ip2}
 \la  q_m(X_1), q_n(X_1)\ra_{a, N}=0\qquad\hbox{if $m\neq n$}.
 \end{equation}
 Thus $q_m(X_1)$ is a monic polynomial of degree $m$ that is orthogonal to the polynomials  $q_0(X_1), \ldots, q_{m-1}(X_1)$, for $m\geq 1$. Moreover, since each $q_n(X_1)$ is of degree $n$, it follows inductively that $q_0(X_1), \ldots, q_n(X_1)$ form a basis of the space $\mcp^{\leq n}_1$ of polynomials in $X_1$ of degree $\leq n$. 
 
Thus, using now the notation $q_m(a;X_1)$ instead of $q_m(X_1)$ to stress the role of $a$, we have
 \begin{equation}\label{E:qmX1}
q_m(a;X_1)= \left(I-{\dot\Pi}^{\leq m}_{1,N}\right)X_1^m,
\end{equation}
where
\begin{equation} \label{E:defdotPiN}
{\dot\Pi}^{\leq m}_{1,N}: \mcp^{\leq m}_{1}\to\mcp^{\leq m-1}_{1}
\end{equation}
is the orthogonal projection within the finite-dimensional space $\mcp^{\leq m}_1$, with respect to the inner-product $\la\cdot,\cdot\ra_{a,N}$.   

The inner-product of two polynomials in $X_1$ is given by
\begin{equation}\label{E:ippqX1}
\begin{split}
\la p(X_1), q(X_1)\ra_{a,N} &= \frac{c_{N-2}}{c_{N-1}a^{N-1}   } a\int_{-a}^a p(x)\overline{q(x)} (a^2-x^2)^{\frac{N-3}{2}}\,dx\\
&= \frac{c_{N-2}}{c_{N-1}a    }  \int_{-a}^a p(x)\overline{q(x)} \left(1-\frac{x^2}{a^2}\right)^{\frac{N-3}{2}}\,dx
\end{split}
\end{equation}
where we have used the disintegration formula (\ref{E:spheredisint3}) with $k=1$ and $d=N-1$, and the denominator reflects the normalization by the `surface area' of $S^{N-1}(a)$.  
Thus we have established the following result.

\begin{prop}\label{P:orthoqmX}
The zonal harmonic polynomials $q_0(X), q_1(X), \ldots$ are obtained by Gram-Schmidt orthogonalization of $1, X, X^2, \ldots$ relative to the inner-product given by 
\begin{equation}\label{E:pqaNip}
\la p, q\ra_{a,N}  = \frac{c_{N-2}}{c_{N-1}a    }  \int_{-a}^a p(x)\overline{q(x)} \left(1-\frac{x^2}{a^2}\right)^{\frac{N-3}{2}}\,dx.
\end{equation}
\end{prop} 
Note that here, by the Gram-Schmidt process, we mean the process (\ref{E:qmX1}) of taking the orthogonal projection onto the orthogonal complement of the span of the previously obtained basis vectors (we do not normalize).

\subsection{Disintegration}\label{ss:dis} Here we review a disintegration formula that is useful when integrating functions over a sphere.
Let $S^d(a)$ be the sphere of radius $a$, center $0$, in $\mbr^{d+1}$. For $k\in\{1,\ldots, d-1\}$ and $x\in \mbr^k$ we have the slice $\{y\in\mbr^{d+1-k}: (x,y)\in S^d(a)\}$, which is a sphere of radius 
\begin{equation}
a_x=\sqrt{a^2-\norm{x}^2}.
\end{equation}
Thus the slice is $S^{d-k}(a_x)$. 
The volume of this slice is
\begin{equation}
a_x^{d-k}c_{d-k},
\end{equation}
where $c_j$ is the surface area of the $j$-dimensional unit sphere:
\begin{equation}\label{E:cjsurf}
c_{j}= 2\frac{\pi^{\frac{j+1}{2}}}{\Gamma\left(\frac{j+1}{2}\right)}.
\end{equation}
We use the following disintegration formula that expresses the integral of a function $f$ on the sphere $S^d(a)$  of radius $a$, center $0$, in $\mbr^{d+1}$:
 \begin{equation}\label{E:spheredisint2}
\int_{S^d(a)}f\,d\sigma =\int_{x\in B_k(a)}\left[\int_{y\in S^{d-k}(a_{x})} f(x,y)\,d\sigma(y)\right]\frac{a}{a_x}\,dx,
\end{equation}
where $S^d(a)$ is the sphere of radius $a$, centered at $0$, in $\mbr^{d+1}$, and $B_k(a)$ is the ball of radius $a$, center $0$, in $\mbr^k$ (for a proof see \cite{SenGRL2018}). If $f$ depends only on $x\in \mbr^k$ then we have
 \begin{equation}\label{E:spheredisint3}
\int_{S^d(a)}f(x)\,d\sigma(x,y) =c_{d-k}a\int_{x\in B_k(a)} f(x) a_x^{d-k-1}\,dx.
\end{equation}

 \section{Hermite limits for monomials over large spheres}\label{s:ol}

In this section we show that monomials, suitably projected, over the sphere $S^{N-1}(\sqrt{N})$ converge to  Hermite polynomials.

\subsection{The subspaces $\mcp^d$ and projections $\Pi^d$}\label{ss:subproj}
 We equip  the space $\mcp$ of all polynomials in variables $X_1, X_2, \ldots$ with the Gaussian $L^2$ inner-product:
 \begin{equation}
 \la p(X_1,\ldots, X_N), q(X_1,\ldots, X_N)\ra=\la p, q\ra_{L^2(\gamma_N)}
 \end{equation}
 where $\gamma_N$ is the standard Gaussian measure on $\mbr^N$:
 $$d\gamma_N(x)=(2\pi)^{-N/2}e^{-\norm{x}^2/2}\,dx.$$
 Each space $\mcp^{\leq d}_N$ is finite-dimensional and there is an orthogonal projection
 \begin{equation}
 \Pi^{\leq d}_N:\mcp^{\leq d}_N\to\mcp^{\leq d-1}_N,
 \end{equation}
 for all $d\geq 1$ and $N\geq 1$. We can drop the subscript $N$ from $\Pi_{d,N}$ because  of the following observation.
 
 \begin{lemma}\label{L:PidNM}
If $M>N$ then
 \begin{equation}\label{E:PiMN}
  \Pi^{\leq d}_M|\mcp^{\leq d}_N = \Pi^{\leq d}_N.
  \end{equation}
  \end{lemma}
   A consequence of this equality is that $  \Pi^{\leq d}_M p(X_1,\ldots, X_N)$ {\em is a polynomial in} $X_1,\ldots, X_N$, as seems natural.
  \begin{proof}
Consider any $p\in \mcp^{\leq  d}_N$, and any monomial $X_1^{j_1}\ldots X_M^{j_M}\in \mcp^{\leq d-1}_M$; then:
\begin{equation}
\begin{split}
&\la p(X_1,\ldots, X_N), X_1^{j_1}\ldots X_M^{j_M}\ra \\
 &=\la  p(X_1,\ldots, X_N), X_1^{j_1}\ldots X_N^{j_N}\ra \la 1, X_{N+1}^{j_{N+1}}\ldots X_M^{j_M}\ra,\end{split}
\end{equation}
where we have used the fact that the Gaussian measure $\gamma_{M}$ is the product of the standard Gaussian measure $\gamma_N$ and the standard Gaussian measure in the remaining $M-N$ variables:
\begin{equation}
\begin{split}
 \int_{\mbr^{N}\times\mbr^{M-N}} f(x)g(y)\,d\gamma_M(x, y)  & =\int_{\mbr^N}f\,d\gamma_N\int_{\mbr^{M-N}}g\,d\gamma_{M-N}
\end{split}
\end{equation}
Next we observe that
 \begin{equation}
 \la p(X_1,\ldots, X_N), X_1^{j_1}\ldots X_N^{j_N}\ra 
 =\la    \Pi^{\leq d}_N p(X_1,\ldots, X_N), X_1^{j_1}\ldots X_N^{j_N}\ra,\end{equation}
 by definition of the orthogonal projection $\Pi^{\leq d}_N$, keeping in mind that   $ X_1^{j_1}\ldots X_N^{j_N}$ has degree $\leq d-1$.  
 
 Using the product nature of $\gamma_M$ again we conclude that
\begin{equation}\label{E:pPiNd}
 \la p(X_1,\ldots, X_N), X_1^{j_1}\ldots X_M^{j_M}\ra =\la \Pi^{\leq d}_N p(X_1,\ldots, X_N), X_{1}^{j_{1}}\ldots X_M^{j_M}\ra, 
\end{equation}
for all monomials $ X_{1}^{j_{1}}\ldots X_M^{j_M}$ in $\mcp^{\leq d-1}_{M}$. On the right, $ \Pi^{\leq d}_N p(X_1,\ldots, X_N)$ is in $\mcp^{\leq d-1}_N$. 
Since the monomials in $\mcp^{\leq d-1}_M$ form a basis of this space, we conclude from (\ref{E:pPiNd}) that 
\begin{equation}
\Pi^{\leq d}_M p(X_1,\ldots, X_N)= \Pi^{\leq d}_N p(X_1,\ldots, X_N),
\end{equation}
which establishes (\ref{E:PiMN}). \end{proof}

Thus all the linear mappings $\Pi^{\leq d}_N$ combine to form one linear mapping
\begin{equation}
\Pi^{\leq d} :\mcp^{\leq d}\to\mcp^{\leq d-1}.
\end{equation}

\subsection{Hermite polynomials as orthogonal projections of monomials}
Now let 
\begin{equation}
\Pi^{\leq d}_\perp  =I-\Pi^{\leq d}.
\end{equation}
Then
\begin{equation}\label{E:PidNperp}
\Pi^{\leq d}_\perp|\mcp^{\leq d}_N \quad\hbox{is the orthogonal projection onto $\mcp^{\leq d}_N\ominus\mcp^{\leq d-1}_N$.}
\end{equation}
 
Focusing for the moment on just one variable $X$, the $m$-th Hermite polynomial $H_m(X)$ is $1$ if $m=0$ and is the projection of $X^m$ to the subspace orthogonal to $\mcp^{m-1}$ otherwise. Thus
\begin{equation}\label{E:HmX}
H_m(X)= \Pi^{\leq m}_{\perp}(X^m).
\end{equation}
We note that
$$X^m-H_m(X)$$
is the orthogonal projection of $X^m$ on the subspace of polynomials of degree $<m$, and so is of degree $<m$.

We have next the generalization of this observation to more than one variable:

\begin{prop}\label{P:go} For any monomial $X_1^{m_1}\ldots X_N^{m_N}$ with $m_1+\ldots +m_N=d$, we have
\begin{equation}\label{E:PidXH}
\Pi^{\leq d}_{\perp}(X_1^{m_1}\ldots X_N^{m_N})=H_{m_1}(X_1)\ldots H_{m_N}(X_N).
\end{equation}
 \end{prop}
\begin{proof} Writing $X_1^{m_1}\ldots X_N^{m_N}$ as
\begin{equation}
\begin{split}
&X_1^{m_1}\ldots X_N^{m_N}\\
&=\bigl(X_1^{m_1}-H_{m_1}(X_1)\bigr)X_2^{m_2}\ldots X_N^{m_N}+H_{m_1}(X_1)X_2^{m_2}\ldots X_N^{m_N},
\end{split}
\end{equation}
we observe  that the first term on the right hand side is of total degree $<d$, because $X_1^{m_1}-H_{m_1}(X_1)$ is of degree $m_1-1$ in $X_1$. Therefore, applying the projection $\Pi^d_{\perp}$, we have
\begin{equation}
\Pi^{\leq d}_{\perp}\bigl(X_1^{m_1}\ldots X_N^{m_N}\bigr)= 
\Pi^{\leq d}_{\perp}\bigl(H_{m_1}(X_1)X_2^{m_2}\ldots X_N^{m_N}\bigr).
\end{equation}
Repeating this argument with $X_2, X_3, ...$, we obtain
\begin{equation}\label{E:repeat}
\Pi^{\leq d}_{\perp}\bigl(X_1^{m_1}\ldots X_N^{m_N}\bigr)= 
\Pi^{\leq d}_{\perp}\bigl(H_{m_1}(X_1)\ldots H_{m_N}(X_N)\bigr).
\end{equation}
Now we check that the polynomial $H_{m_1}(X_1)\ldots H_{m_N}(X_N)$ is orthogonal to all polynomials of total degree $<d$: if $X_1^{j_1}\ldots X_N^{j_N}$ is a monomial of total degree $<d$ then $j_k<m_k$ for at least one $k$, and so
\begin{equation}
\begin{split}
&\la H_{m_1}(X_1)\ldots H_{m_N}(X_N), X_1^{j_1}\ldots X_N^{j_N}\ra\\
&=\la H_{m_1}(X_1), X_1^{j_1}\ra\ldots  \la H_{m_N}(X_N), X_N^{j_N}\ra\\
&=0\quad\hbox{because $\la H_{m_k}(X_k), X_k^{j_k}\ra=0$.}
\end{split}
\end{equation}
Hence
$$\Pi^{\leq d}_{\perp}\bigl(H_{m_1}(X_1)\ldots H_{m_N}(X_N)\bigr)=H_{m_1}(X_1)\ldots H_{m_N}(X_N),$$
and so, by (\ref{E:repeat}), the result (\ref{E:PidXH}) follows.
\end{proof}

\subsection{The limiting orthogonal projection}\label{ss:proj}
We turn now to look at orthogonal projections associated to a sequence of inner-products. We will apply the following result to the case of inner-products given by integration over $S^{N-1}(\sqrt{N})$.

\begin{prop}\label{P:projl}
Let $V$ be a finite-dimensional vector space, and $\la\cdot,\cdot\ra_n$ an inner-product on $V$ for each $n\in\{1,2,3,\ldots\}$, and suppose that there is an inner-product $\la\cdot,\cdot\ra$ on $V$ which is the limit of the sequence of inner-products $\la\cdot,\cdot\ra_n$. Let $P_n:V\to V$ be the orthogonal projection onto a subspace $W\subset V$ relative to the inner-product $\la\cdot,\cdot\ra_n$. Then $P_nv\to Pv$, as $n\to\infty$, for all $v\in V$, where $P$ is the orthogonal projection onto $W$ with respect to the inner-product $\la\cdot,\cdot\ra$.
\end{prop}
\begin{proof} Let $w_1,\ldots, w_k$ form a basis of $W$. Fix $v\in V$. Then, writing
\begin{equation}\label{E:Pnv}
P_nv=\sum_{j=1}^kc_j(n)w_j,
\end{equation}
we have
\begin{equation}\label{E:Pnv2}
\la w_i, v\ra_n=\la w_i, P_nv\ra_n=\sum_{j=1}^k\la w_i, w_j\ra_n c_j(n).
\end{equation}
Hence the vector $c(n)$ whose components are $c_1(n), \ldots, c_k(n)$ is given by matrix inversion:
\begin{equation}
c(n)=\left[\begin{matrix}\la w_1, w_1\ra_n &\la w_1, w_2\ra_n &\ldots &\la w_1, w_k\ra_n\\
\la w_2, w_1\ra_n &\la w_2, w_2\ra_n &\ldots &\la w_2, w_k\ra_n\\
\vdots & \vdots & \vdots & \vdots   \\
\la w_k, w_1\ra_n &\la w_k, w_2\ra_n &\ldots &\la w_k, w_k\ra_n
\end{matrix}\right]^{-1}\left[\begin{matrix}\la w_1, v\ra_n \\  \la w_2, v\ra_n \\  \vdots\\  \la w_k, v\ra_n \end{matrix}\right].
\end{equation}
Now we let $n\to\infty$ and using continuity of matrix inversion and matrix multiplication we obtain:
\begin{equation}
c\stackrel{\rm def}{=}\lim_{n\to\infty} c(n)=\left[\begin{matrix}\la w_1, w_1\ra &\la w_1, w_2\ra  &\ldots &\la w_1, w_k\ra \\
\la w_2, w_1\ra  &\la w_2, w_2\ra  &\ldots &\la w_2, w_k\ra \\
\vdots & \vdots & \vdots & \vdots   \\
\la w_k, w_1\ra  &\la w_k, w_2\ra  &\ldots &\la w_k, w_k\ra 
\end{matrix}\right]^{-1}\left[\begin{matrix}\la w_1, v\ra  \\  \la w_2, v\ra  \\  \vdots\\  \la w_k, v\ra  \end{matrix}\right].
\end{equation}
Then, multiplying by the matrix with entries $\la w_i, w_j\ra$, we have:
\begin{equation}\label{E:cwp}
\sum_{j=1}^k\la w_i, w_j\ra c_j =\la w_i, v\ra,
\end{equation}
and this means that $Pv=\sum_{j=1}^k c_jw_j$. Looking back at (\ref{E:Pnv}) we conclude that $\lim_{n\to\infty}P_nv=Pv$.
\end{proof}

\subsection{Hermite polynomials from monomials on spheres}\label{ss:herm}

Recall that $\mcp^{\leq d}_k$ is the vector space of all polynomials of degree $\leq d$ in the variables $X_1,\ldots, X_k$.

\begin{prop}\label{P:limPiN}Let 
\begin{equation}\label{E:defdotPikN}
{\dot\Pi}^{\leq d}_{k,N}: \mcp^{\leq d}_k\to\mcp^{\leq d-1}_{k}
\end{equation}
  be the orthogonal projection using the inner-product $ \la \cdot, \cdot\ra_N=\la\cdot,\cdot\ra_{L^2(S^{N-1}(\sqrt{N}), \ovs), }$. Then
\begin{equation} \lim_{N\to\infty}(I-{\dot\Pi}^{\leq d}_{k,N})(X_1^{j_1}\ldots X_k^{j_k})  = H_{j_1}(X_1)\ldots H_{j_k}(X_k).
\end{equation}
\end{prop}
\begin{proof}  By Theorem  \ref{T:ipNinfty}   the limit of the inner-product $\la\cdot,\cdot\ra_N$ as $N\to\infty$ is the Gaussian inner-product. Then by Proposition  \ref{P:projl} we have the limit of the projections:
\begin{equation}
\lim_{N\to\infty}{\dot\Pi}^{\leq d}_{k,N}(X_1^{j_1}\ldots X_k^{j_k})=\Pi^{\leq d}(X_1^{j_1}\ldots X_k^{j_k}).
\end{equation}
Hence
\begin{equation}
\begin{split}
\lim_{N\to\infty}(I-{\dot\Pi}^{\leq d}_{k,N})(X_1^{j_1}\ldots X_k^{j_k})&=(I-\Pi^{\leq d})(X_1^{j_1}\ldots X_k^{j_k})\\
&=\Pi^{\leq d}_{\perp}(X_1^{j_1}\ldots X_k^{j_k}).
\end{split}
\end{equation} 
Then by Proposition \ref{P:go}, we have:
\begin{equation}\label{E:Piherm}
\begin{split}
\lim_{N\to\infty}(I-{\dot\Pi}^{\leq d}_{k,N})(X_1^{j_1}\ldots X_k^{j_k}) &=  
  H_{j_1}(X_1)\ldots H_{j_k}(X_k).
\end{split}
\end{equation}
\end{proof}

\subsection{Hermite polynomials from zonal harmonics}\label{ss:hermzon}  Consider the special case $k=1$ in (\ref{E:Piherm}).  Recall from (\ref{E:defdotPiN}) that the polynomial $q_m(X_1)$, giving rise to a zonal harmonic,  is the orthogonal projection onto the orthogonal complement of the subspace spanned by the monomials $1, X_1, X_1^2, \ldots, X_1^{m-1}$. Thus  (\ref{E:Piherm}) specializes to:
\begin{equation}
\lim_{N\to\infty}q_m(\sqrt{N}; X)= H_m(X).
\end{equation}
Using (\ref{E:qmaXGgn}) we can write this as a limiting relationship between Gegenbauer polynomials and Hermite polynomials:

 \begin{equation}\label{E:GgnHerm}
\lim_{N\to\infty}C^{\left((N-2)/2\right)}_m(X/\sqrt{N})= H_m(X).
\end{equation}

\section{Limit of the Spherical Laplacian}\label{s:lapl}

In this section we introduce the spherical Laplacian $\Delta_{S^{N-1}(a)}$ using a special algebraic property of $\norm{\mathbf M}^2$ and determine the limiting behavior of $\Delta_{S^{N-1}}(\sqrt{N})$ as $N\to\infty$.

\subsection{The Laplacian for $S^{N-1}(a)$}\label{ss:sphlap} The spherical Laplacian corresponding to the sphere $S^{N-1}(a)$  should emerge from the simple scaling
$$\frac{1}{a^2}\sum_{\{j,k\}\in P_2(N) }M_{jk}^2.$$
However, we will refine this notion before formulating   a definition. We define, for $p\in  \mbc[X_1,\ldots, X_n]$,
\begin{equation}\label{E:sphlapnNa}
 \Delta_{S^{N-1} (a)}p= \frac{1}{a^2}\sum_{\{j,k\}\in P_2(N) }M_{jk}^2p\in \mbc[X_1,\ldots, X_n]\qquad\hbox{if $n\geq N$.}
 \end{equation}
Now let $N>n$; then
$$q= \frac{1}{a^2}\sum_{\{j,k\}\in P_2(N) }M_{jk}^2p\in \mbc[X_1,\ldots, X_N].$$
We will show in Proposition \ref{P:Mjkfmin}  that $q$ is equal, modulo $\mcz_N(a)$, to a unique polynomial  $q_{\rm min}$ in $X_1,\ldots, X_n$, and we define
 \begin{equation}\label{E:sphlapnNa2}
 \Delta_{S^{N-1} (a)}p=q_{\rm min} \in \mbc[X_1,\ldots, X_n]\qquad\hbox{if $N>n$.}
   \end{equation}
   Linearity of $\sum_{\{j,k\}\in P_2(N) }M_{jk}^2p$ in $p$ and the uniqueness of $q_{\rm min}$ implies that $\Delta_{S^{N-1}(a)}$ is a linear operator.
   
   \begin{prop}\label{P:Deltaspherelinsa} The mapping $p\mapsto  \Delta_{S^{N-1} (a)}p$ specifies a linear operator on each finite-dimensional space $\mcp^{\leq d}_n$ of polynomials in $X_1,\ldots, X_n$ of total degree $\leq d$.  Moreover, this operator is self-adjoint with respect to the inner-product $\la\cdot,\cdot\ra_{a,N}$ on $\mcp^{\leq d}_n$, for $1\leq n <N$, given in (\ref{E:ippq}).   \end{prop}   
  
  \begin{proof}
  We have already explained linearity of $\Delta_{S^{N-1} (a)}$. Self-adjointness follows from the relation (\ref{E:skewherm}) along with the fact that any polynomial in $\mcz_N(a)$ vanishes pointwise as a function on $S^{N-1}(a)$.
  \end{proof}

\subsection{$ \Delta_{S^{N-1} (\sqrt{N})}$ on monomials}\label{ss:sphlapmon}  From (\ref{E:M1j2X2}) we have 
\begin{equation}\label{E:laplx1}
\begin{split}
 \sum_{\{j,k\}\in P_2(N)}M_{jk}^2(X_1^m)& = -m(N-1)X_1^m+  m(m-1)(N-X_1^2)X_1^{m-2}\\
&\qquad\hbox{modulo $\mcz_N(\sqrt{N})$}.
\end{split}
\end{equation}  
Hence
\begin{equation}\label{E:DSphNX1n}
 \Delta_{S^{N-1} (\sqrt{N})}(X_1^m) = -m\left(1-\frac{1}{N}\right)X_1^m +m(m-1)\left(1-\frac{X_1^2}{N}\right)X_1^{m-2}.
 \end{equation}
 
 \subsection{$ \Delta_{S^{N-1} (\sqrt{N})}$ on products}\label{ss:sphlapprod} 
Now consider a polynomial $f\in \mbc[X_1,\ldots, X_n]$ of the form
\begin{equation}
f  =pq,
\end{equation}
where $p\in\mbc[X_1,\ldots, X_m] $ and $q\in\mbc[X_{m+1},\ldots, X_n] $. Then, since $M_{jk}$ is a first order differential operator, we have
\begin{equation}\label{E:Mjksqrdf}
\begin{split}
  M_{jk}^2f  
&=\bigl(M_{jk}^2p\bigr)q   +2M_{jk}p M_{jk}q  +p M_{jk}^2q.
\end{split}
\end{equation}
The second term on the right hand side is zero unless either $(j,k)$ or $(k,j)$ lies in $\{1,\ldots, m\}\times\{m+1,\ldots, n\}$. So we have the sum
\begin{equation}
\begin{split}
  \sum_{\{j,k\}\in\mcp_2(N)} M_{jk}p M_{jk}q  
&=\sum_{j=1}^m\sum_{k=m+1}^n\bigl(-X_k\partial_jp\bigr)\bigl(X_j\partial_kq \bigr)\\
&=-\sum_{j=1}^nX_j\partial_jp \sum_{k=m+1}^nX_k\partial_kq \\
&=-d_pd_qpq,
\end{split}
\end{equation}
if $p$ is homogeneous of degree $d_p$ and $q$ is homogeneous of degree $d_q$. Thus, working with the case where $p$ and $q$ are both {\em monomials} (and hence homogeneous), we have
\begin{equation}\label{E:Mjksqrdfpqd}
\begin{split}
&   \sum_{\{j,k\}\in\mcp_2(N)}  M_{jk}^2f  \\
&=\left[  \sum_{\{j,k\}\in\mcp_2(N)} M_{jk}^2p\right] q   -2 d_pd_q f   +p \left[   \sum_{\{j,k\}\in\mcp_2(N)} M_{jk}^2q \right].
\end{split}
\end{equation}

\begin{prop}\label{P:Mjkfmin} Let $f\in \mbc[X_1,\ldots, X_n]$,  $N>n$ and $a>0$. Then $  \sum_{\{j,k\}\in\mcp_2(N)}  M_{jk}^2f$ is equivalent, modulo $\mcz_N(a)$, to a unique polynomial in $\mbc[X_1,\ldots, X_n]$.
\end{prop}
\begin{proof} The uniqueness statement follows from Proposition \ref{P:reprsnt}. To prove existence it will suffice to assume that $f$ is a monomial because the general polynomial $f$ is a linear combination of monomials. If $f$ is of the form $X_i^{n}$ then (\ref{E:laplx1}), { applied to $X_i$ instead of $X_1$},  shows that $  \sum_{\{j,k\}\in\mcp_2(N)}  M_{jk}^2f$ is equivalent, modulo $\mcz_N(a)$, to a   polynomial in $\mbc[X_i]$. Finally, the general statement follows inductively from (\ref{E:Mjksqrdf}),by taking $p$ to be a factor of $f$ of the form $X_i^{n_i}$ and $q=f/p$.\end{proof}


{\bf Remark:} A referee has pointed out that (\ref{E:x2Dintro}) implies that the sum $\sum_{\{j,k\}\in P_2(N)}M_{jk}^2f(X_1,\ldots, X_n)$ coincides, as a {\em function}, on $S^{N-1}(a)$ with a polynomial function in $X_1, \ldots, X_n$. It should be noted that it is the {\em uniqueness} part of the result that is useful to us in defining $\Delta_{S^{N-1}(a)}$ and establishing its linearity. 

The right hand side of  (\ref{E:Mjksqrdfpqd}) is equal, modulo $\mcz_N(a)$, to
\begin{equation}
\left[\Delta_{S^{N-1}(a)}p\right]q-2d_pd_qpq+p\Delta_{S^{N-1}(a)}q
\end{equation}
and this is a polynomial in $X_1, \ldots, X_n$. Thus
\begin{equation}\label{E:Deltpq}
\Delta_{S^{N-1}(a)}(pq)=\left[\Delta_{S^{N-1}(a)}p\right]q-2\frac{d_pd_q}{a^2}pq+p\Delta_{S^{N-1}(a)}q,
\end{equation}
where $p$ is a homogeneous polynomial in $X_1,\ldots, X_m$ of degree $d_p$, and $q$ is a homogeneous polynomial in $X_{m+1}, \ldots, X_n$ of degree $d_q$.
   
 \subsection{Limiting spherical Laplacian on polynomials}\label{ss:sphpoly}    
From (\ref{E:DSphNX1n}) we have the limit:
\begin{equation}\label{E:laplx1Nlim}
\begin{split}
  \lim_{N\to\infty}  \Delta_{S^{N-1}(\sqrt{N})}X_1^m  
&= m(m-1)X_1^{m-2}-mX_1^m\\
&=\left[ \frac{d^2}{dX_1^2} -X_1\frac{d}{dX_1}\right]X_1^m.
\end{split}
\end{equation}
On the left we are taking the limit, as $N\to\infty$, of a sequence of polynomials in $X_1$ of degree $\leq m$. This is a limit in a finite-dimensional vector space and so is meaningful without any additional topological structure.

\begin{prop}\label{P:limSLappolydiffvar}
If $p$ and $q$ are polynomials in different sets of variables, then
\begin{equation}\label{E:Mjksqrdfpqdlim}
\begin{split}
 \lim_{N\to\infty}\Delta_{S^{N-1}(\sqrt{N})}   (pq) 
&=  \left[ \lim_{N\to\infty}\Delta_{S^{N-1}(\sqrt{N})}   p\right] q     +p \left[ \lim_{N\to\infty}\Delta_{S^{N-1}(\sqrt{N})}   q \right].
\end{split}
\end{equation}
\end{prop}

\begin{proof}
Let $p$ be a homogeneous polynomial in $X_1, \ldots, X_m$ of degree $d_p$ and $q$ be a homogeneous polynomial in $X_{m+1}, \ldots, X_n$ of degree $d_q$. 
Using (\ref{E:Deltpq}) we have:

\begin{equation}\label{E:Mjksqrdfpqdlim2}
\begin{split}
&\lim_{N\to\infty}\Delta_{S^{N-1}(\sqrt{N})}   (pq) \\
&= \lim_{N\to\infty}\left(  \left[\Delta_{S^{N-1}(\sqrt{N})}   p\right] q   -2
 \frac{d_pd_q }{N}pq  +p \left[\Delta_{S^{N-1}(\sqrt{N})}   q \right]\right).
\end{split}
\end{equation}
 
As $N\to\infty$, the middle term on the right hand side goes to $0$.

\end{proof}

Applying Proposition \ref{P:limSLappolydiffvar}   inductively, with $p$ and $q$ being monomials, and using (\ref{E:laplx1Nlim}), we have
\begin{equation}\label{E:limDeltasum}
 \lim_{N\to\infty}   \Delta_{S^{N-1}(\sqrt{N})}\left(X_1^{j_1}\ldots X_m^{j_m}\right)
 = \sum_{i=1}^m\left[ \frac{\partial^2}{{\partial}X_i^2}-X_i\frac{\partial}{{\partial}X_i}\right]\left(X_1^{j_1}\ldots X_m^{j_m}\right).
\end{equation}
Let $\ch$ be the Hermite operator:
\begin{equation}
\ch=\sum_{j=1}^\infty\left[ \frac{\partial^2}{{\partial}X_j^2}-X_j\frac{\partial}{{\partial}X_j}\right],
\end{equation}
acting on polynomials. 

We  can now summarize our observations:
\begin{prop}\label{P:limspjLapl}
With notation as above, for any polynomial $p\in \mcp^{\leq m}_N$ the elements $\Delta_{S^{N-1}(\sqrt{N})}p$ all lie in the finite-dimensional space $\mcp^{\leq m}_N$, and
\begin{equation}\label{E:limDeltasumch}
 \lim_{N\to\infty}   \Delta_{S^{N-1}(\sqrt{N})}p={\ch}p.
 \end{equation}
 
\end{prop}
Thus, following  Umemura and Kono \cite{Ume1965}, one may think of the operator $\ch$ as an infinite-dimensional  (spherical) Laplacian.


{\bf Remark:} It might seem that equation (\ref{E:limDeltasumch}) can be obtained by applying the identity (\ref{E:x2Dintro}) to a polynomial $p$ restricted as a function to the sphere $S^{N-1}(\sqrt{N})$:
\begin{equation}\label{E:fall}
\begin{split}
&N\Delta_{N}p(x_1,\ldots, x_n) \\&=  [(r\partial_r)^2+(N-2)r\partial_r]p(x_1,\ldots, x_n) +\norm{{\bf M}}^2p(x_1,\ldots, x_n),   
\end{split}
\end{equation}
where $n\leq N$; the dividing by $N$ and letting $N\to\infty$ would seem to yield (\ref{E:limDeltasumch}). However, there is a fallacy in this argument.  In (\ref{E:fall}) the point $x=(x_1,\ldots, x_n)$ lies on the sphere $S^{N-1}(\sqrt{N})$, and so it makes no sense to let $N\to\infty$ while `holding $x$ fixed.' In our Proposition \ref{P:limspjLapl} the limit is not in the sense of a limit of a sequence of functions on the spheres $S^{N-1}(\sqrt{N})$ but as the limit of a sequence of algebraic polynomials.

\subsection{Orthogonal projections and limits}\label{ss:ortlimproj}
Recall from (\ref{E:defdotPiN}) the orthogonal projection
$${\dot\Pi^{\leq d}}_{N,k}:\mcp^{\leq d}_k\to\mcp^{\leq d-1}_k,$$
projecting onto the space of polynomials in $X_1,\ldots, X_k$ of total degree $<d$.  The role of $N$ here is to provide the inner-product $\la\cdot,\cdot\ra_N$ relative to which the orthogonal projection is being taken. The spaces involved are all finite-dimensional.

 Then
\begin{equation}
\Delta_{S^{N-1}(\sqrt{N})}  {\dot\Pi^{\leq d}}_{N,k}={\dot\Pi^{\leq d}}_{N,k} \Delta_{S^{N-1}(\sqrt{N})}
\end{equation}
This is because $\Delta_{S^{N-1}(\sqrt{N})} $ is a self-adjoint operator (Proposition \ref{P:Deltaspherelinsa}) on $\mcp^{\leq d}_N$, with respect to the inner-product $\la\cdot,\cdot\ra_N$, and preserves the subspace $\mcp_k^{\leq d}$.  Then, working entirely within the finite-dimensional space $\mcp_k^{\leq d}$,  we have:
\begin{equation}
\begin{split}
 \lim_{N\to\infty}\Delta_{S^{N-1}(\sqrt{N})}  {\dot\Pi^d}_{N,k}p   
&=\lim_{N\to\infty}  {\dot\Pi^{\leq d}}_{N,k} \Delta_{S^{N-1}(\sqrt{N})}p  \\
&=\lim_{N\to\infty}{\dot\Pi^{\leq d}}_{N,k}\lim_{N\to\infty} \Delta_{S^{N-1}(\sqrt{N})}p  \\
&=\Pi^{\leq d}{\ch}p.
\end{split}
\end{equation}
On the other hand, the left side is equal to $\mch \Pi^{\leq d}$. Thus
\begin{equation}\label{E:chPi}
\mch \Pi^{\leq d}=\Pi^{\leq d}{\ch}.
\end{equation}

{\bf{ Acknowledgments.}} This research was  supported in part by by NSA grant H98230-16-1-0330.   

\bibliographystyle{plain}

\end{document}